\documentclass[a4paper,11pt,reqno]{amsart}

\usepackage[headings]{fullpage}

\usepackage{dsfont,amssymb,graphicx,mathdots}
\usepackage{tikz-cd}

\usepackage{hyperref}

\usepackage[backend = biber,
                maxnames = 5,
				sortcites  = true,
                giveninits = true,
                doi        = false,
                isbn       = false,
                url        = false,
                style      = alphabetic]{biblatex}
\DeclareNameAlias{sortname}{family-given}
\DeclareNameAlias{default}{family-given}
\renewbibmacro{in:}{%
  \ifentrytype{article}{}{\printtext{\bibstring{in}\intitlepunct}}}

\usepackage{wasysym}
\newcommand{\circt}%
{\mathbin{%
\mathchoice
{\ooalign{$\ocircle$\cr\hidewidth\raise-.15ex\hbox{$\scriptstyle\top\mkern2.05mu$}\cr}}% Woronowicz style tensor product, USUAL SIZE
{\ooalign{$\ocircle$\cr\hidewidth\raise-.15ex\hbox{$\scriptstyle\top\mkern2.05mu$}\cr}}% Woronowicz style tensor product, USUAL SIZE
{\ooalign{$\scriptstyle\ocircle$\cr\hidewidth\raise-.12ex\hbox{$\scriptscriptstyle\top\mkern1mu$}\cr}}% Woronowicz style tensor product, SCRIPT SIZE
{\ooalign{$\scriptstyle\ocircle$\cr\hidewidth\raise-.12ex\hbox{$\scriptscriptstyle\top\mkern1mu$}\cr}}% Woronowicz style tensor product, SCRIPT SIZE
}}

\DeclareFontFamily{U}{matha}{\hyphenchar\font45}
\DeclareFontShape{U}{matha}{m}{n}{
      <5> <6> <7> <8> <9> <10> gen * matha
      <10.95> matha10 <12> <14.4> <17.28> <20.74> <24.88> matha12
      }{}
\DeclareSymbolFont{matha}{U}{matha}{m}{n}
\DeclareFontFamily{U}{mathb}{\hyphenchar\font45}
\DeclareFontShape{U}{mathb}{m}{n}{
      <5> <6> <7> <8> <9> <10> gen * mathb
      <10.95> mathb10 <12> <14.4> <17.28> <20.74> <24.88> mathb12
      }{}
\DeclareSymbolFont{mathb}{U}{mathb}{m}{n}
\DeclareMathSymbol{\ovoid}{3}{matha}{"6C}
\DeclareMathSymbol{\boxvoid}{2}{mathb}{"6C}

\mathchardef\mhyph="2D

\numberwithin{equation}{section}

\newtheorem{theorem}{Theorem}[section]
\newtheorem{corollary}[theorem]{Corollary}
\newtheorem{lemma}[theorem]{Lemma}
\newtheorem{proposition}[theorem]{Proposition}
\theoremstyle{remark}
\newtheorem{remark}[theorem]{Remark}
\newtheorem{example}[theorem]{Example}
\theoremstyle{definition}
\newtheorem{definition}[theorem]{Definition}

\newcommand\bp{\begin{proof}}
\newcommand\ep{\end{proof}}
\newcommand\ee{\nopagebreak\mbox{\ }\hfill$\diamond$}

\DeclareMathOperator{\Ad}{Ad}
\newcommand\Corep{\operatorname{Corep}}

\DeclareMathOperator{\Hilbf}{\operatorname{Hilb}_{\mathrm f}}
\DeclareMathOperator{\Mat}{Mat}
\DeclareMathOperator{\Mor}{Mor}

\DeclareMathOperator{\Rep}{Rep}

\DeclareMathOperator{\Tr}{Tr}

\DeclareMathOperator{\Hom}{Hom}
\newcommand\Ch{\operatorname{Ch}}
\DeclareMathOperator{\Vectf}{\operatorname{Vect}_{\mathrm f}}

\newcommand{\C}{{\mathbb C}}

\newcommand{\Z}{{\mathbb Z}}

\newcommand{\R}{{\mathbb R}}
\newcommand\T{{\mathbb T}}

\newcommand{\A}{{\mathcal A}}
\newcommand{\B}{{\mathcal B}}

\newcommand\CC{{\mathcal C}}

\newcommand\E{\mathcal E}
\newcommand{\F}{{\mathcal F}}
\newcommand\HH{\mathcal H}

\newcommand{\K}{{\mathcal K}}

\newcommand\OO{\mathcal O}

\newcommand\TT{\mathcal T}

\newcommand\id{\mathrm{id}}
\newcommand\eps{\varepsilon}

\newcommand{\GL}{\mathrm{GL}}

\newcommand{\GA}{\tilde{O}_P^+}
\newcommand{\GC}{\tilde{O}_Q^+}
\newcommand{\FP}{\mathcal{F}_P}
\newcommand{\TP}{\mathcal{T}_P}
\newcommand{\OP}{\mathcal{O}_P}

\begin{document}

\title[Subproduct systems II]{Subproduct systems with quantum group symmetry. II}

\author{Erik Habbestad}
\address{University of Oslo, Mathematics institute}
\email{erikhab@math.uio.no}

\author{Sergey Neshveyev}
%\address{Universitetet i Oslo}
\email{sergeyn@math.uio.no}

\thanks{Supported by the NFR project 300837 ``Quantum Symmetry''.}

\begin{abstract}
We complete our analysis of the Temperley--Lieb subproduct systems, which define quantum analogues of Arveson's $2$-shift, by extending the main results of the previous paper to the general parameter case. Specifically, we show that the associated Toeplitz algebras are nuclear, find complete sets of relations for them, prove that they are equivariantly $KK$-equivalent to $\C$ and compute the $K$-theory of the associated Cuntz--Pimsner algebras. A key role is played by quantum symmetry groups, first studied by Mrozinski, preserving Temperley--Lieb polynomials up to rescaling, and their monoidal equivalence to~$U_q(2)$. In Appendix we show how to adapt Voigt's arguments for $SU_q(2)$ to establish the Baum--Connes conjecture for the dual of $U_q(2)$, which is needed in our analysis of $K$-theory.
\end{abstract}


\date{December 16, 2022; revised February 19, 2025}

\maketitle

\section*{Introduction}

Subproduct systems of finite dimensional Hilbert spaces and associated Toeplitz algebras appear in several areas of operator algebras. From the point of view of noncommutative geometry, their study can be seen as doing function theory on algebraic subsets of a noncommutative unit ball~\cite{AP,Shalit-Solel}. Specifically, one studies $m$-tuples of operators $(S_1,\dots,S_m)$ that satisfy the inequality $\sum_iS_iS_i^*\le 1$ and a system of noncommutative homogeneous polynomial equations $P(S_1,\dots,S_m)=0$.

In the previous paper~\cite{HN21}, motivated by recent work of Andersson~\cite{andersson} and Arici--Kaad~\cite{Arici-Kaad}, we began to analyze a class of \emph{Temperley--Lieb} subproduct systems, where we are given just one noncommutative quadratic polynomial $P=\sum^m_{i,j=1}a_{ij}X_iX_j$ such that the matrix~$A\bar A$ is unitary up to a scalar factor. We succeeded in obtaining fairly detailed information about the associated C$^*$-algebras under the more restrictive assumption $A\bar A=\pm1$. In particular, we showed that the corresponding Toeplitz algebras are nuclear, found complete sets of relations for them, proved that these C$^*$-algebras are $KK$-equivalent to $\C$ and computed the $K$-theory of the associated Cuntz--Pimsner algebras. One of the key points of~\cite{HN21} was to exploit properties of the quantum group $O^+_P$ of transformations leaving~$P$ invariant. The quantum groups $O^+_P$ form one of the most studied classes of compact quantum groups in the operator algebraic literature and are known as free orthogonal quantum groups.

For general Temperley--Lieb subproduct systems the quantum groups $O^+_P$ are usually too small to be very useful, in the sense that the canonical representation of $O^+_P$ on $\C^m$ is reducible. But we can consider the quantum group $\GA$ of transformations that leave $P$ invariant up to a phase factor. It is a peculiar feature of noncommutativity that this quantum group can be much larger than $O^+_P$: it turns out that the representation of $\GA$ on $\C^m$ is irreducible exactly for the Temperley--Lieb subproduct systems. The quantum groups $\GA$ were introduced and studied by Mrozinski~\cite{Mrozinski}, but so far they have received much less attention than the free orthogonal quantum groups. As a consequence we need to prove a number of results for $\GA$ that were readily available for $O^+_P$, but once this is done, the strategy developed in~\cite{HN21} allows us to analyze the Temperley--Lieb subproduct systems in full generality.

\smallskip

In more detail, the contents of the paper are as follows. In Section~\ref{sec:monoidal}, after a brief review of compact quantum groups, we study monoidal equivalence between them and compare its implementation by the so called bi-Galois objects in the purely algebraic, $*$-algebraic and C$^*$-algebraic settings. The first such comparison was carried out by Bichon~\cite{bichon-gal}. Later the theory was developed from scratch in the C$^*$-algebraic setting by Bichon, De Rijdt and Vaes~\cite{BDRV}. We complement the results of these papers by showing that there is no much difference between the three settings. In particular, we show that if we are given a $*$-bi-Galois object, then it is always possible to slightly modify the $*$-structure in order to get a $*$-algebra admitting a C$^*$-completion (Theorem~\ref{thm:Galois-objects}). Furthermore, whether such a modification is needed at all can be easily seen from relations in the algebra (Corollary~\ref{cor:star-criterion}). We need these results only for the quantum groups $\GA$, but hopefully they will be useful in other contexts as well.

\smallskip

In Section~\ref{sec:TL} we recall definition of the Temperley--Lieb subproduct systems and their associated Toeplitz and Cuntz--Pimsner algebras $\TT_P$ and $\OO_P$. Then, following~\cite{Mrozinski}, we formally introduce the quantum groups $\GA$. By combining results of Section~\ref{sec:monoidal} and~\cite{Mrozinski} we obtain C$^*$-$U_q(2)$-$\GA$-Galois objects $B(U_q(2),\GA)$, where $U_q(2)$ is the $q$-deformation of the unitary group $U(2)$ for an appropriate $q\in(0,1]$ depending on~$P$. In~\cite{HN21} we showed that in the free orthogonal case we have an isomorphism $B(SU_q(2),O^+_P)\cong\OO_P$. In the general case the situation is a bit more complicated: we have an isomorphism  $B(U_q(2),\GA)\cong\OO_P\rtimes_{\bar\beta}\Z$ for some gauge-type automorphism $\bar\beta$ (Theorem~\ref{thm:CP-crossed}). This is still enough to find a complete set of relations for $\OO_P$ and then for $\TT_P$ (Theorem~\ref{thm:relations}).

In the simplest cases
$$
P=X_1X_2-qX_2X_1\quad (q\in\C^\times)
$$
this shows that the C$^*$-algebras $\OO_P$ are the function algebras on the (braided for complex~$q$) quantum groups $SU_{\bar q}(2)$. In hindsight it is very natural that we get a deformation of $SU(2)$, since for $q=1$ we have Arveson's symmetric subproduct system $SSP_2$ and the associated Cuntz--Pimsner algebra is $C(S^3)\cong C(SU(2))$~\cite{arveson}. More generally, in view of monoidal equivalences $\GA\sim_\otimes U_q(2)$ and the categorical nature of Toeplitz and Cuntz--Pimsner algebras~\cite{HN21}, the algebras $\OO_P$ for arbitrary Temperley--Lieb subproduct systems can be viewed as quantum relatives of $C(S^3)$.

It is also worth mentioning that the isomorphism $B(U_q(2),\GA)\cong\OO_P\rtimes_{\bar\beta}\Z$ provides a new route to the technical core of Mrozinski's work and therefore can be used to give an alternative proof of the monoidal equivalence $\GA\sim_\otimes U_q(2)$ (Remark~\ref{rem:connected}). In our opinion this alone makes the Temperley--Lieb subproduct systems interesting from the quantum group theoretic perspective.

\smallskip

In Section~\ref{sec:K-theory-T} we prove that the embedding $\C\to\TT_P$ is an isomorphism in the equivariant $KK$-category $KK^{\GA}$ (Theorem~\ref{thm:KK-Toeplitz}). In the free orthogonal case studied in~\cite{HN21} we used the Baum--Connes conjecture for the dual of $SU_q(2)$ to prove a similar result. This conjecture, established by Voigt~\cite{Voigt-BC}, says roughly that arbitrary separable $SU_q(2)$-C$^*$-algebras can be built out of C$^*$-algebras with trivial action. In this paper we need a similar result for $U_q(2)$. We do not see a quick way of deducing this from the case of $SU_q(2)$, but fortunately most of Voigt's arguments upgrade from $SU_q(2)$ to $U_q(2)$ using general considerations, without going into technical details of the proof in~\cite{Voigt-BC}. In order to maintain the flow of the exposition, we defer details to Appendix~\ref{app:A}. Once we accept the Baum--Connes conjecture for the dual of~$U_q(2)$, the proof of the $KK^{\GA}$-equivalence between $\C$ and $\TT_P$ is similar to~\cite{HN21}, although the computations become a bit more involved.

\smallskip

In Section~\ref{sec:K-theory-O} we compute the $K$-theory of the C$^*$-algebras $\OO_P$ (Corollary~\ref{cor:K-theory-O}). For this we first find an inverse of the embedding map $\C\to\TT_P$ in $KK^{\GA}$. The construction is an adaptation of that due to Arici--Kaad~\cite{Arici-Kaad} for a particular class of Temperley--Lieb subproduct systems, which is in turn inspired by a well-known construction for the usual Toeplitz--Pimsner algebras. It is easy to see that the class that we define provides a one-sided inverse to $\C\to\TT_P$ in~$KK^{\GA}$. But this is enough, since we already know from Section~\ref{sec:K-theory-T} that an inverse exists. This circumvents a direct but long and involved argument in~\cite{Arici-Kaad}, which probably can be adapted to arbitrary Temperley--Lieb subproduct systems. Once an inverse is found, the computation of the $K$-theory of $\OO_P$ is a simple application of the six-term exact sequence in $K$-theory.

\bigskip

\section{Compact quantum groups and monoidal equivalence}\label{sec:monoidal}

For an introduction to compact quantum groups and C$^*$-tensor categories we refer the reader to \cite{NT13}, and for more details about actions of compact quantum groups to~\cite{DC}. We will briefly recall the main definitions and key facts.

\medskip

Let $\mathcal{A}$ be a Hopf $*$-algebra with coproduct $\Delta\colon\mathcal{A}\to \mathcal{A}\otimes \mathcal{A}$, counit $\varepsilon\colon \mathcal{A}\to\C$ and antipode $S\colon\mathcal{A}\to \mathcal{A}$. Recall that a right $\mathcal{A}$-comodule consists of a vector space $H$ and a linear map $\delta\colon H \to H \otimes \mathcal{A}$ satisfying
\[(\iota\otimes\Delta)\delta = (\delta\otimes\iota)\delta \quad \mbox{ and } \quad (\iota\otimes\varepsilon)\delta = \iota, \]
where $\iota$ denotes the identity map. We write $\Corep_f\mathcal{A}$ for the $\C$-linear tensor category of finite dimensional $\mathcal{A}$-comodules. An object $(H,\delta) \in \Corep_f\mathcal{A}$ is the same thing as an element $U \in B(H) \otimes \mathcal{A}$ satisfying
\[(\iota\otimes\Delta)(U) = U_{12}U_{13}\quad  \mbox{ and } \quad (\iota\otimes\varepsilon)(U) = 1.\]
The corresponding map $\delta\colon H \to H \otimes \mathcal{A}$ is then given by $\delta(\xi) = U(\xi\otimes 1)$. The comodule $(H,\delta)$ is said to be unitary if $H$ is a Hilbert space and the element $U \in B(H) \otimes \mathcal{A}$ is unitary. Elements of the form $((\,\cdot\, \xi, \zeta)\otimes \iota)(U) \in \mathcal{A}$, for $\xi,\zeta \in H$, are called matrix coefficients of $U$ (or $\delta$).

\begin{definition}
A \emph{compact quantum group} $G$ is a Hopf $*$-algebra $(\C[G], \Delta, \varepsilon, S)$ that is spanned by matrix coefficients of finite dimensional unitary comodules.
\end{definition}

A finite dimensional unitary representation of $G$ is by definition a finite dimensional unitary $\C[G]$-comodule. We write $\Rep G$ for the C$^*$-tensor category of finite dimensional \textit{unitary} representations of $G$. A fundamental fact is that the functor $\Rep G \to \Corep_f \C[G]$ forgetting the Hilbert space structures is an equivalence of $\C$-linear tensor categories. Two compact quantum groups $G_1$ and $G_2$ are said to be monoidally equivalent, which we write as $G_1 \sim_\otimes G_2$, if there is a unitary monoidal equivalence between $\Rep G_1$ and $\Rep G_2$.

For every compact quantum group $G$, there is a unique Haar state $h$ on $\C[G]$, and we denote the corresponding GNS-space by $L^2(G)$. The norm closure of $\C[G] \subset B(L^2(G))$ is denoted by~$C(G)$ and is called the (reduced) C$^*$-algebra of functions on $G$. If $C(G)$ is the only C$^*$-completion of $\C[G]$, then $G$ is called coamenable.

\smallskip

A right action by $G$ on a C$^*$-algebra $B$ is a $*$-homomorphism $\alpha\colon B \to B \otimes C(G)$ such that
$$
(\iota\otimes \Delta)\alpha = (\alpha\otimes \iota)\alpha
$$
and there is a dense $*$-subalgebra of $B$ on which $\alpha$ defines a right $\C[G]$-comodule structure. The largest such subalgebra is called the algebra of regular elements, or the algebraic core of $B$, and is denoted by $\mathcal B$. When $G$ acts on $B$, we say also that $B$ is a (right) $G$-C$^*$-algebra.  Left actions and left $G$-C$^*$-algebras are defined analogously.

Given a right action $\alpha$ on $B$, there is always a conditional expectation onto the fixed point algebra $B^G$:
\[E = (\iota\otimes h)\alpha\colon B \to B^G = \{ \, b \in B \, | \, \alpha(b) = b \otimes 1 \, \}. \]
The action is called reduced if $E$ is faithful, equivalently, $\alpha$ is injective. By passing from $B$ to $B/(\ker\alpha)$ we can always go from any action to a reduced one with the same algebraic core. If~$G$ is coamenable, then all actions are reduced.

An action $\alpha$ of $G$ on a unital C$^*$-algebra $B$ is called ergodic if $B^G = \C 1$. In this case the algebraic core of $B$ is the only dense $*$-subalgebra of $B$ on which $\alpha$ defines a $\C[G]$-comodule structure.

A $*$-homomorphism $\phi\colon B \to C$ between $G$-C$^*$-algebras is called equivariant if $\alpha_C \circ \phi = (\phi \otimes \iota) \circ \alpha_B$. A simple but very useful observation is that if the action on $B$ is reduced, then $\phi$ is injective if and only if it is injective on $B^G$. In particular, if the action on $B$ is reduced and ergodic, then $\phi$ is either injective or zero.

\begin{definition}
A unital $G$-C$^*$-algebra $(B,\alpha)$ is called a right \emph{C$^*$-$G$-Galois extension} of $\C$ if the action is ergodic and the map
\[ \mathcal{B} \otimes \mathcal{B} \to \mathcal{B} \otimes \C[G], \quad b \otimes b' \mapsto (b\otimes 1)\alpha(b'), \]
is a linear isomorphism. Such actions are also called \emph{ergodic actions of full quantum multiplicity}.
\end{definition}

One can also consider $*$-$G$-Galois extensions (drop the condition that there is a C$^*$-comple\-tion~$B$ of~$\mathcal{B}$ above), and simply $G$-Galois extensions (forget the $*$-structure as well). Left $G$-Galois extensions are defined similarly.

There is a one-to-one correspondence between the isomorphism classes of left $G$-Galois extensions of $\C$ and the natural monoidal isomorphism classes of fiber functors $\Corep_f\C[G]\to\Vectf$. Namely, given a $G$-Galois extension~$\B$ of~$\C$, the corresponding fiber functor $\F$, with its monoidal structure $\F_2$, is given by
$$
\F(M):=M \boxvoid_{G} \mathcal{B} = \{x \in M \otimes \mathcal{B} \, | \, (\delta \otimes \iota)(x) = (\iota\otimes \alpha)(x)\},
$$
$$
\F_{2;M,N}\colon \F(M)\otimes\F(N)\to\F(M\otimes N),\qquad x\otimes y\mapsto x_{13}y_{23},
$$
for finite dimensional $\C[G]$-comodules $M$ and $N$. Similarly, there is a one-to-one correspondence between the isomorphism classes of reduced left C$^*$-$G$-Galois extensions of $\C$ and the natural unitary monoidal isomorphism classes of unitary fiber functors $\Rep G\to\Hilbf$; we will return to how to define scalar products on $\F(H)$ for finite dimensional unitary $\C[G]$-comodules $(H,\delta)$ in a moment.

\begin{definition}
Assume that $G_1$ and $G_2$ are two compact quantum groups and that $B$ is a C$^*$-algebra with a left action by $G_1$ and a right action by $G_2$: \[ C(G_1)\otimes B \xleftarrow{\delta_1} B \xrightarrow{\delta_2} B \otimes C(G_2).\]
We say that $B$ is a \emph{C$^*$-$G_1$-$G_2$-Galois object}, or a \emph{linking algebra}, if the actions commute and~$B$ is both a left C$^*$-$G_1$-Galois extension of $\C$ and a right C$^*$-$G_2$-Galois extension of $\C$.
\end{definition}

A C$^*$-$G_1$-$G_2$-Galois object exists if and only if $G_1 \sim_{\otimes} G_2$. Again, we can also define $*$-$G_1$-$G_2$-Galois objects and $G_1$-$G_2$-Galois objects.
Existence of a $G_1$-$G_2$-Galois object is equivalent to having an equivalence of the $\C$-linear tensor categories $\Corep_f\C[G_1]$ and $\Corep_f\C[G_2]$. Then $\Rep G_1$ and $\Rep G_2$ are also equivalent as $\C$-linear tensor categories, and the next result implies that in fact $G_1 \sim_{\otimes} G_2$. This is probably known to the experts and the proof is a simple application of polar decomposition, cf.~\cite[Lemma~4.3.3]{NT13}.

\begin{proposition} \label{prop:polar1}
Assume $\CC_1$ and $\CC_2$ are small semisimple C$^*$-tensor categories that are equivalent as $\C$-linear monoidal categories. Then they are unitarily monoidally equivalent.
\end{proposition}

\bp
We may assume that $\CC_1$ and $\CC_2$ are strict. Let $\F\colon\CC_1\to\CC_2$ be an equivalence of the $\C$-linear monoidal categories. If we ignore the tensor structures, then such an equivalence is determined uniquely, up to a natural isomorphism, by a bijection between the isomorphism classes of simple objects, and every such bijection can be used to define a unitary functor. Therefore without loss of generality we may assume that the functor $\F$, without its tensor structure, is already unitary.

Consider the isomorphisms
$$
\F_{2;X,Y}\colon\F(X)\otimes\F(Y)\to\F(X\otimes Y)
$$
defining the tensor structure on $\F$. Using the polar decomposition we can write
$$
\F_{2;X,Y}=\F(\Omega_{X,Y})U_{X,Y}
$$
for uniquely defined unitary morphism $U_{X,Y}\colon \F(X)\otimes\F(Y)\to\F(X\otimes Y)$ and strictly positive morphism
$\Omega_{X,Y}\colon X\otimes Y\to X\otimes Y$. By naturality of $\F_2$ these morphisms are natural in~$X$ and~$Y$.

By naturality of $U$ we have
\begin{align*}
\F_{2;X\otimes Y,Z}(\F_{2;X,Y}\otimes\iota)&=\F(\Omega_{X\otimes Y,Z})U_{X\otimes Y,Z}(\F(\Omega_{X,Y})\otimes\iota)(U_{X,Y}\otimes\iota)\\
&=\F(\Omega_{X\otimes Y,Z}(\Omega_{X,Y}\otimes\iota))U_{X\otimes Y,Z}(U_{X,Y}\otimes\iota).
\end{align*}
Note that by naturality of $\Omega$ the morphisms $\Omega_{X\otimes Y,Z}$ and $\Omega_{X,Y}\otimes\iota$ commute and hence their composition is strictly positive.
For the same reasons we have
$$
\F_{2;X, Y\otimes Z}(\iota\otimes \F_{2;Y,Z})=\F(\Omega_{X,Y\otimes Z}(\iota\otimes\Omega_{Y,Z}))U_{X,Y\otimes Z}(\iota\otimes U_{Y,Z}),
$$
and the morphism $\Omega_{X,Y\otimes Z}(\iota\otimes\Omega_{Y,Z})$ is strictly positive. The identity
$$
\F_{2;X\otimes Y,Z}(\F_{2;X,Y}\otimes\iota)=\F_{2;X, Y\otimes Z}(\iota\otimes \F_{2;Y,Z})
$$
and uniqueness of polar decomposition imply now that
$$
\Omega_{X\otimes Y,Z}(\Omega_{X,Y}\otimes\iota)=\Omega_{X,Y\otimes Z}(\iota\otimes\Omega_{Y,Z}),\qquad
U_{X\otimes Y,Z}(U_{X,Y}\otimes\iota)=U_{X,Y\otimes Z}(\iota\otimes U_{Y,Z}).
$$
Therefore the morphisms $U$ can be used to define a new unitary tensor structure on $\F$.
\ep

In a similar way, using polar decomposition, it is not difficult to prove the following.

\begin{proposition} \label{prop:polar2}
If two unitary tensor functors between C$^*$-tensor categories are naturally monoidally isomorphic, then they are also naturally unitarily monoidally isomorphic. \end{proposition}

Note that in the proof of Proposition~\ref{prop:polar1} the new unitary monoidal equivalence is not necessarily naturally monoidal\-ly isomorphic to the one we started with, since the morphisms $\Omega$ may define a nontrivial tensor structure on the identity functor $\CC_1\to\CC_1$. For compact quantum groups this means that if there is a $G_1$-$G_2$-Galois object $\B$, then there is a also a C$^*$-$G_1$-$G_2$-Galois object, but it may not be possible to introduce a $*$-structure on $\B$ to define such an object by completion. Furthermore, an invariant $*$-structure alone (meaning that the action maps are $*$-preserving) is not enough to have a C$^*$-completion. Let us illustrate these phenomena with the following example.

\begin{example}
Let $\Gamma$ be a discrete group and consider the compact quantum group $G=\hat\Gamma$, so $\Rep G$ is the category of $\Gamma$-graded finite dimensional Hilbert spaces. Every element $g\in\Gamma$ defines a simple object in $\Rep G$ that we denote by $\C_g$. Every normalized cocycle $\omega\in Z^2(\Gamma;\C^\times)$, where normalization means $\omega(e,g)=\omega(g,e)=1$, defines a tensor functor $\Rep G\to \Rep G$ that is the identity functor on the objects and morphisms, but whose tensor structure is given by
$$
\C_g\otimes\C_h\xrightarrow{\omega(g,h)}\C_g\otimes\C_h=\C_{gh}.
$$
It is easy to see that this functor is naturally monoidally isomorphic to a unitary tensor functor if and only if $\omega$ is cohomologous to a $\T$-valued cocycle. For example, if $\Gamma=\Z^2$, then, up to coboundaries, every cocycle has the form $\omega(g,h)=z^{g_1h_2}$ for a uniquely defined $z\in\C^\times$, and it is cohomologous to a $\T$-valued one if and only if $z\in\T$.

Correspondingly, we have a $G$-$G$-Galois object $\C_\omega\Gamma$, the twisted group algebra of $\Gamma$ with generators $u_g$, $g\in\Gamma$, and relations
$$
u_gu_h=\omega(g,h)u_{gh}.
$$
The actions of $G$ are given by
$$
\delta_1\colon\C_\omega\Gamma\to\C\Gamma\otimes\C_\omega\Gamma,\qquad \delta_2\colon\C_\omega\Gamma\to\C_\omega\Gamma\otimes\C\Gamma,
$$
$\delta_1(u_g)=g\otimes u_g$, $\delta_2(u_g)=u_g\otimes g$. From the previous paragraph we can conclude that if $\omega$ is not cohomologous to a $\T$-valued cocycle, then $\C_\omega\Gamma$ does not complete to a C$^*$-$G$-$G$-Galois object, equivalently, as $\delta_1$ and $\delta_2$ are essentially the same, $(\C_\omega\Gamma,\delta_1)$ does not complete to a C$^*$-$G$-Galois extension of $\C$. But, in fact, it is not difficult to check directly that then there are no $*$-structures on $\C_\omega\Gamma$ such that $\delta_1$ is $*$-preserving.

If $\omega$ is $\T$-valued, then the standard $*$-structure on $\C_\omega\Gamma$ admitting a C$^*$-completion is given by $u_g^*=u_{g^{-1}}$. It is again not difficult to check that the $*$-structures for which $\delta_1$ is $*$-preserving are given by
$$
u_g^*=\chi(g)u_{g^{-1}}\quad\text{for a homomorphism}\quad\chi\colon\Gamma\to\R^\times.
$$
If $\chi(g)<0$ for some $g$, then $\C_\omega\Gamma$ with such a $*$-structure does not admit any nonzero $*$-representation on a Hilbert space, since $u_g^*u_g=\chi(g)1$.\ee
\end{example}

Let $G$ be a compact quantum group and $(\B,\delta)$ be a left $*$-$G$-Galois extension of $\C$. Consider the corresponding fiber functor
$\F\colon\Corep_f\C[G]\to\Vectf$. If $U\in B(H_U)\otimes\C[G]$ is a finite dimensional unitary representation of $G$, recall that we can view $H_U$ as a right  $\C[G]$-comodule by $\delta_U(\xi)=U(\xi\otimes1)$. Define a Hermitian form $\langle\cdot,\cdot\rangle$ on~$\F(H_U)$~by
\begin{equation}\label{eq:Hermitian}
\langle x,y\rangle 1=y^*x\in{}^{G}\,\,\!\!\B=\C1,
\end{equation}
where we view vectors in $H_U$ as operators $\C\to H_U$. In other words, if $(\xi_i)_i$ is an orthonormal basis in $H_U$, $x=\sum_i\xi_i\otimes x_i$ and $y=\sum_i\xi_i\otimes y_i$, then $\langle x,y\rangle1=\sum_iy_i^*x_i$. Note that given two unitary representations $U$ and $V$, the maps $\F_{2;H_U,H_V}$ are isometric with respect to the Hermitian forms we defined.

\smallskip

The following observation is implicit in the proof of~\cite[Proposition~4.3.1]{bichon-gal}.

\begin{lemma}\label{lem:Bichon}
The Hermitian form $\langle\cdot,\cdot\rangle$ on $\F(H_U)$ is nondegenerate.
\end{lemma}

\bp
Consider the right comodule $M=\bar H_U$ with the coaction $\delta(\bar\xi):=\overline{\xi_{(0)}}\otimes\xi_{(1)}^*$, where $\delta_U(\xi)=\xi_{(0)}\otimes\xi_{(1)}$. This comodule is a right dual to $H_U$, with the standard evaluation map given by $\phi\colon M\otimes H_U\to\C$, $\phi(\bar\xi\otimes\zeta):=(\zeta,\xi)$. It follows that $\F(M)$ is a right dual to $\F(H_U)$ in  $\Vectf$, with the duality given by the evaluation map $\F(\phi)\F_{2;M,H_U}$. This is equivalent to saying that the pairing
$$
(M\boxvoid_G\B)\times(H_U\boxvoid_G\B)\to\C1,\quad (\sum_i\bar\xi_i\otimes y_i,\sum_i\xi_i\otimes x_i)\mapsto\sum_i y_ix_i,
$$
where $(\xi_i)_i$ is an orthonormal basis in $H_U$, is nondegenerate. Combined with the antilinear isomorphism $H_U\boxvoid_G\B\cong M\boxvoid_G\B$, $\sum_i\xi_i\otimes y_i\mapsto \sum_i\bar\xi_i\otimes y_i^*$, this gives the result.
\ep

This eventually leads to the following criteria for existence of C$^*$-completions.

\begin{theorem}[\cite{bichon-gal}]\label{thm:Bichon}
Assume $G$ is a compact quantum group and $\B$ is a left $*$-$G$-Galois extension of~$\C$. Then the following conditions are equivalent:
\begin{enumerate}
  \item[(1)] $\B$ admits a C$^*$-completion;
  \item[(2)] $\B$ has a state, that is, a linear functional $\phi$ such that $\phi(1)=1$ and $\phi(a^*a)\ge0$ for all~$a\in \B$;
  \item[(3)] the nondegenerate Hermitian form on $H_U\boxvoid_G\B$ given by~\eqref{eq:Hermitian} is positive definite for all finite dimensional unitary representations $U$ of $G$.
\end{enumerate}
\end{theorem}

Note that $\B$ may have many C$^*$-completions, but if it has one, it has a unique completion such that we get a reduced action of $G$ on it.

\bp
Equivalence of (1) and (2) follows from \cite[Theorem~5.2.1 and Proposition~4.2.5]{bichon-gal}. That (2) implies (3) follows from Lemma~\ref{lem:Bichon} and is already proved in~\cite[Proposition~4.3.1]{bichon-gal}. That (3) implies (1) follows from~\cite[Theorem~4.3.4]{bichon-gal} and its proof; see in particular Step 3 there and ~\cite[Proposition 4.3.3]{bichon-gal}, which show how the $*$-structure on $\B$ can be reconstructed from the scalar products~\eqref{eq:Hermitian} defining a unitary structure on the fiber functor corresponding to $\B$.
\ep

\begin{remark}\label{rem:iso*}
If $*_1$ and $*_2$ are two invariant $*$-structures on a $G$-Galois extension $\B$ of $\C$ admitting a C$^*$-completion, then there is a $G$-equivariant automorphism $\beta$ of $\B$ such that $\beta(a^{*_1})=\beta(a)^{*_2}$ for all $a\in \B$. Indeed, both involutions define unitary structures on the fiber functor associated with $\B$. By Proposition~\ref{prop:polar2}, the two unitary fiber functors we thus get are naturally unitarily monoidally isomorphic. By~\cite[Theorem 4.3.4]{bichon-gal} it follows that the $*$-$G$-Galois extensions $(\B,*_1)$ and $(\B,*_2)$ are isomorphic. \ee
\end{remark}

In general a $G$-Galois extension of $\C$ can have a large group of automorphisms and many invariant $*$-structures.

\begin{example}
Let $G$ be a compact Lie group. The algebra $\A=\C[G]$ is a left $G$-Galois extension of $\C$ corresponding to the forgetful functor $\Corep_f\C[G]\to\Vectf$. The group $G$ is a real affine algebraic group and $\C[G]$ is the algebra of regular functions on its complexification $G_\C$. Hence the $G$-equivariant automorphisms of $\A$ are in a one-to-one correspondence with the right translations on $G_\C$ by elements $g\in G_\C$. Such an automorphism preserves the standard $*$-structure on $\A=\C[G]$ if and only if $g\in G$. By Remark~\ref{rem:iso*} it follows that the $G$-invariant $*$-structures on $\A$ admitting a C$^*$-completion can be parameterized by the points of the coset space $G_\C/G$.

Every such $*$-structure $\star$ defined by $[g]\in G_\C/G$ can be further twisted by a nontrivial $G$-equivariant $\star$-preserving involutive automorphism $\beta$ by letting $a^\#=\beta(a)^\star$. Such automorphisms~$\beta$ are, in turn, parameterized by the order $2$ elements of $gGg^{-1}\cong G$. \ee
\end{example}

Let us now turn to bi-Galois objects, which have much smaller automorphism groups. We need a bit of preparation to formulate our results.

\smallskip

Recall that the chain group $\Ch(\CC)$ of a small rigid C$^*$-tensor category $\CC$ is the group generated by the formal symbols $[X]$ for the simple objects $X$ of $\CC$, subject to the relations $[X][Y ] = [Z]$ whenever $Z$ appears as a subobject of $X\otimes Y$. Recall also that a Hopf algebra homomorphism $\pi$ is called cocentral if $(\pi\otimes\iota)\Delta=(\pi\otimes\iota)\Delta^{\mathrm{op}}$. If $G$ is a compact quantum group and $\Gamma=\Ch(\Rep G)$, then we get a cocentral homomorphism $\pi\colon\C[G]\to\C\Gamma$ of Hopf $*$-algebras, which is universal among such homomorphisms into group algebras, see~\cite[Section 1.1]{BNY16}. Explicitly, if $u_{ij}$ is a matrix coefficient of an irreducible unitary representation $U$, then $\pi(u_{ij})=\eps(u_{ij})[U]$.

It is clear from the definition that if $\F\colon\CC_1\to\CC_2$ is a monoidal equivalence of tensor categories, then $\F$ induces an isomorphism $\Ch(\CC_1)\cong\Ch(\CC_2)$.

Now, if $(\B,\delta)$ is a left $G$-Galois extension of $\C$, then every quasi-character $\chi\colon\Gamma\to\C^\times$ defines a $G$-equivariant automorphism $\alpha_\chi$ of $\B$ by
$$
\alpha_\chi(a):=(\chi\pi\otimes\iota)\delta(a), \quad a \in \B.
$$
In other words, $\alpha_\chi$ acts on the spectral subspace of $\B$ corresponding to an irreducible unitary representation $U$ by multiplication by the scalar $\chi([U])$. If $\chi$ is real-valued and we have a $G$-invariant $*$-structure on $\B$, then we can define a new one by
$$
a^\star:=\alpha_\chi(a)^*.
$$
We thus get an action of $\Hom(\Gamma,\R^\times)$ on the set of $G$-invariant $*$-structures on $\B$.

\begin{theorem}\label{thm:Galois-objects}
Assume $G_1$ and $G_2$ are compact quantum groups, $\Gamma=\Ch(\Rep G_1)$ and~$\B$~is a $G_1$-$G_2$-Galois object. Then either $\B$ does not admit any $G_1$-$G_2$-invariant $*$-structures or the set of such structures forms a torsor over the group $\Hom(\Gamma,\R^\times)$. In the latter case the subset of $*$-structures admitting a C$^*$-completion is nonempty and forms a torsor over the group $\Hom(\Gamma,\R^\times_+)$.
\end{theorem}

Recall that by a torsor over $\Gamma$, or a $\Gamma$-torsor, one means a set equipped with a free transitive action of $\Gamma$.

In the above formulation we can of course equally well use the chain group of $\Rep G_2$, which is isomorphic to $\Gamma$ thanks to the monoidal equivalence between $\Rep G_1$ and $\Rep G_2$ defined by $\B$. Let us also note that the entire group of $G_1$-$G_2$-equivariant automorphisms of $\B$ is isomorphic to $\Hom(\Gamma,\C^\times)$, which is an equivalent form of~\cite[Proposition~1.2]{BNY16}, but we will not need this explicitly.

\bp[Proof of Theorem~\ref{thm:Galois-objects}]
Assume there is a $G_1$-$G_2$-invariant involution $*$ on $\B$. Consider the monoidal functor $\F\colon\Rep G_1\to\Corep_f\C[G_2]$ defined by $\B$. Observe that if $U$ is an irreducible unitary representation of $G_1$, then, on the one hand, the Hermitian form $\langle\cdot,\cdot\rangle$ on $\F(H_U)$ defined by~\eqref{eq:Hermitian} is $G_2$-invariant. On the other hand, the $(\C[G_2],\Delta)$-comodule $\F(H_U)$ is simple, hence it admits an invariant scalar product and every other invariant Hermitian form on $\F(H_U)$ is a scalar multiple of it. Therefore $\langle\cdot,\cdot\rangle$ is either positive definite or negative definite on $\F(H_U)$. Let $\chi([U])=1$ in the first case and $\chi([U])=-1$ in the second. Then $\chi$ defines a homomorphism $\Gamma\to\{\pm1\}$.
Consider the involution $\star$ defined by $a^\star:=\alpha_\chi(a)^*$. By construction the Hermitian forms on $\F(H_U)$ defined by $\star$ are positive definite. By Theorem~\ref{thm:Bichon}(3), $(\B,\star)$ admits a C$^*$-completion.

Next, if $*_1$ and $*_2$ are two $G_1$-$G_2$-invariant involutions on $\B$ admitting C$^*$-completions, then for irreducible $U$ the scalar product on $\F(H_U)$ defined by~$*_2$ must be a positive scalar multiple of that defined by~$*_1$. These scalars define an element $\chi\in\Hom(\Gamma,\R^\times_+)$ such that the involution~$\star$
given by $a^\star:=\alpha_\chi(a)^{*_1}$ defines the same scalar products on $\F(H_U)$ as $*_2$. Since an invariant involution on a $G_1$-Galois extension of $\C$ can be recovered from the associated Hermitian forms, see again the proofs of~\cite[Proposition~4.3.3 and Theorem~4.3.4]{bichon-gal}, we conclude that $*_2=\star$.

Finally, if $*$ is a $G_1$-$G_2$-invariant involution on $\B$ admitting a C$^*$-completion, $\chi\in\Hom(\Gamma,\R^\times_+)$ and $\star$ is defined by $a^\star:=\alpha_\chi(a)^*$, then $\star$ also admits a C$^*$-completion. This follows either from Theorem~\ref{thm:Bichon}(3) or by observing that $\alpha_{\chi^{1/2}}(a^\star)=\alpha_{\chi^{1/2}}(a)^*$.

It follows that the group $\Hom(\Gamma,\R^\times)$ acts transitively on the set of invariant $*$-structures and $\Hom(\Gamma,\R^\times_+)$ acts transitively on the nonempty subset of $*$-structures admitting a C$^*$-completion. Since the action of $\Hom(\Gamma,\R^\times)$ is obviously free, this gives the result.
\ep

As a corollary we get the following algebraic characterization of C$^*$-bi-Galois objects.

\begin{corollary}\label{cor:star-criterion}
Assume $G_1$ and $G_2$ are compact quantum groups and $(\B,\delta_1,\delta_2)$~is a $*$-$G_1$-$G_2$-Galois object. Assume there exist finite dimensional unitary representations $U_i\in B(H_i)\otimes\C[G_1]$, finite dimensional Hilbert spaces~$K_i$  and elements $X_i\in B(K_i,H_i)\otimes\B$ such that $\B$ is generated as a $*$-algebra by the matrix coefficients of $X_i$ ($i\in I$) and
$$
X_iX_i^*=1,\quad (\iota\otimes\delta_1)(X_i)=(U_i)_{12}(X_i)_{13}\quad\text{for all}\quad i.
$$
Then $\B$ admits a C$^*$-completion.
\end{corollary}

Note that conversely, by~\cite[Theorem~5.1.1]{bichon-gal}, if a $*$-$G_1$-$G_2$-Galois object $\B$ admits a C$^*$-completion, then for every irreducible unitary representation $U$ of $G_1$ there is a unitary element $X^U\in B(K_U,H_U)\otimes\B$, where $K_U=H_U\boxvoid_{G_1}\B$, such that $(\iota\otimes\delta_1)(X^U)=U_{12}X^U_{13}$, and then $\B$ is spanned by the matrix coefficients of such elements $X^U$.

\bp[Proof of Corollary~\ref{cor:star-criterion}]
By decomposing $U_i$ into irreducible representations we may assume that~$U_i$~are already irreducible. Let $\Gamma$ be the chain group of $\Rep G_1$. By Theorem~\ref{thm:Galois-objects} there is a character $\chi\colon\Gamma\to\{\pm1\}$ such that the involution $\star$ defined by $a^\star:=\alpha_\chi(a)^*$ admits a C$^*$-completion. Since the matrix coefficients of $X_i$ lie in the spectral subspaces corresponding to~$U_i$, we have
$X_i X_i^\star=\chi([U_i])X_i X_i^*=\chi([U_i])1$. Hence $\chi([U_i])=1$ for all $i$.
%Since $\A$ is generated as a $*$-algebra by the matrix coefficients of $X_i$ and every spectral subspace of $\A$ is nonzero, the group $\Gamma$ is generated by the elements $[U_i]$. Hence $\chi$ is trivial and $\star=*$.
Therefore the involutions~$*$ and~$\star$ coincide on the matrix coefficients of $X_i$, hence also on the matrix coefficients of $X_i^*$. Since the matrix coefficients of~$X_i$ and~$X_i^*$ generate the algebra $\B$, we conclude that~$*=\star$.
\ep

\bigskip

\section{Temperley--Lieb subproduct systems and associated \texorpdfstring{C$^*$}{C*}-algebras}\label{sec:TL}

We start this section by recalling the basic ingredients of subproduct systems. For more details see, e.g., \cite{Shalit-Solel}.

By a subproduct system we will mean a family of finite dimensional Hilbert spaces $\mathcal{H} = \{H_n\}_{n \in \Z_+}$, $\dim H_0=1$, together with isometries $w_{m,n}\colon H_{m+n} \to H_m \otimes H_n$ satisfying
\[ (w_{m,n}\otimes \iota) w_{m+n,k} = (\iota\otimes w_{n,k}) w_{m,n+k} \colon H_{m+n+k} \to H_m \otimes H_n \otimes H_k\]
for all $m,n,k \in \Z_+$. We remind that $\iota$ denotes the identity map. The Fock space associated to~$\mathcal{H}$ is the Hilbert space $\mathcal{F}_\mathcal{H} := \bigoplus_{n\ge0} H_n$. We can define operators on $\mathcal{F}_{\mathcal{H}}$ by
\[ S_\xi(\zeta) := w_{1,n}^*(\xi\otimes \zeta), \quad \xi \in H_1,\ \zeta \in H_n.\]
The Toeplitz algebra associated to $\mathcal{H}$ is \[\mathcal{T}_{\mathcal{H}} = C^*(1_\mathcal{F},S_1,S_2,...,S_m) \subset B(\mathcal{F}),\] where $S_i = S_{\xi_i}$ for an orthonormal basis $(\xi_i)_{i=1}^m$ in $H_1$.  It is straightforward to verify that $1_{\mathcal{F}} - \sum_i S_iS_i^*$ is the projection onto $H_0$, and it follows that the compacts $\K(\mathcal{F}_\mathcal{H})$ are contained in $\mathcal{T}_{\mathcal{H}}$. Thus we may also define the associated Cuntz-Pimsner algebra: \[\mathcal{O}_{\mathcal{H}} := \mathcal{T}_{\mathcal{H}}/\K(\mathcal{F}_\mathcal{H}).
\]

A subproduct system is called standard if $H_0=\C$, $H_{m+n}  \subset H_{m} \otimes H_{n}$ and $w_{m,n}$ are the embedding maps. In this case we have
\[ \mathcal{F}_{\mathcal{H}} \subset \bigoplus^\infty_{n=0} H_1^{\otimes n} \quad \mbox{ and } \quad S_\xi(\zeta) = f_{n+1}(\xi\otimes \zeta), \quad \xi \in H_1, \zeta \in H_n, \]
where $f_{n+1}$ is the projection $f_{n+1} \colon H_1^{\otimes (n+1)} \to H_{n+1}$.

Given a homogeneous ideal $I \subset \C\langle X_1,X_2,...,X_m \rangle$ in the algebra of noncommutative polynomials in $m$ variables, containing no nonzero constant and linear polynomials, we get a standard subproduct  system with $H_1 = \C^m$ by setting
\begin{equation*}
\label{eq:ideal->subsys}
%H_0 = \C, \quad H_1 = \C^m, \quad
H_n = I_n^\perp \cap (\C^m)^{\otimes n},
\end{equation*}
where $I_n \subset I$ is the homogeneous component of degree $n$ and we identify $\C\langle X_1,X_2,...,X_m \rangle$ with the tensor algebra $T(\C^m)$. All standard subproduct systems are obtained this way, see \cite[Proposition 7.2]{Shalit-Solel}. In the present paper we will mostly consider standard subproduct systems generated by special quadratic polynomials.

\begin{definition}[\cite{HN21}]
Let $H$ be a finite dimensional Hilbert space of dimension $m \geq 2$. A nonzero vector $P \in H\otimes H$ is called \emph{Temperley--Lieb} if there is $\lambda > 0$ such that the orthogonal projection $e\colon H\otimes H \to \C P$ satisfies
\begin{equation}\label{eq:TL}
(e\otimes 1)(1\otimes e)(e\otimes 1) = \dfrac{1}{\lambda}(e\otimes 1).
\end{equation}
The standard subproduct system $\mathcal{H}_P = \{H_n\}_n$ defined by the ideal $\langle P\rangle\subset T(H)$ generated by $P$ is called a Temperley--Lieb subproduct system. We write $\mathcal{F}_P = \mathcal{F}_{\mathcal{H}_P}$, $\mathcal{T}_P = \mathcal{T}_{\mathcal{H}_P}$ and $\mathcal{O}_P = \mathcal{O}_{\mathcal{H}_P}$.
\end{definition}

%The name comes from the fact that $e$ gives rise to a representation of the Temperley-Lieb algebras:

\begin{example}
Consider $m=2$ and $P=X_1X_2-X_2X_1$. This is a Temperley--Lieb polynomial: relation~\eqref{eq:TL} is satisfied with $\lambda=4$. The quotient $\C\langle X_1,X_2 \rangle/\langle X_1X_2-X_2X_1\rangle$ is the polynomial algebra $\C[x_1,x_2]$ in commuting variables $x_1,x_2$. Its completion, the Fock space $\F_{X_1X_2-X_2X_1}$, is known as the Drury--Arveson space $H^2_2$. By \cite[Theorem~5.7]{arveson} the corresponding Cuntz--Pimsner algebra $\OO_{X_1X_2-X_2X_1}$ is isomorphic to $C(S^3)$, where $S^3$ is viewed as the unit sphere in~$\C^2$. \ee
\end{example}

There are obvious notions of an isomorphism of subproduct systems~\cite[Definition~1.4]{Shalit-Solel} and, in particular, of an automorphism of such a system. For a standard subproduct system $\HH=(H_n)_{n\in\Z_+}$, its automorphism group can be identified with the subgroup of the unitary group $U(H_1)$ consisting of unitaries $U\colon H_1\to H_1$ such that $U^{\otimes n}H_n=H_n$ for all $n\ge2$. The automorphism group of $\HH_P$ is often rather small, so we will instead consider a quantum automorphism group. This quantum group was introduced by Mrozinski~\cite{Mrozinski} in a different context. We rephrase his definition as follows.

\begin{definition} \label{def:GA}
For $A  = (a_{ij})_{i,j} \in \mathrm{GL}_m(\C)$ and $P = \sum_{i,j} a_{ij}X_iX_j$, define $\C[\GA]$ as the universal unital $*$-algebra generated by a unitary element~$d$ and elements $v_{ij}$, $1 \leq i,j \leq m$ such that
\[V = (v_{ij})_{i,j}\ \ \text{is unitary and}\ \ VAV^t = dA.\]
This is a Hopf-$*$-algebra with comultiplication
\[ \Delta(d) = d\otimes d, \quad \Delta(v_{ij}) = \sum_{k} v_{ik}\otimes v_{kj}.\]
%Given a polynomial $P = \sum_{i,j} a_{ij}X_iX_j$, with $A=(a_{ij})_{i,j}$ invertible, we let $\GA=\tilde{O}^+_{A^t}$.
\end{definition}

We remark that in the notation of Mrozinski the Hopf $*$-algebra $\C[\GA]$ is $\mathcal{G}(\bar{A},A)$ or $A_{\tilde o}(\bar A)$. In the dimension $2$, for  $q > 0$ and
$$
P=q^{-1/2}X_1X_2-q^{1/2}X_2X_1,
$$
the quantum group $\GA$ coincides with the $q$-deformation $U_q(2)$ of the unitary group $U(2)$ (see, e.g., \cite[\S9.2.4]{KS}).

Note that the second defining relation of $\GA$ can be written as $V^t=A^{-1}V^*A\,d$, and by applying transpose we equivalently get
\begin{equation}\label{eq:Vc}
V=A^tV^c(A^t)^{-1}d,\quad\text{or}\quad V^c=(A^t)^{-1}VA^td^*,
\end{equation}
where $V^c=(v^*_{ij})_{i,j}$.

Viewing $V$ as an element of $\Mat_m(\C)\otimes\C[\GA]$, by definition of $\GA$ we have
\begin{equation*} \label{eq:poly-symmetry}
V_{13}V_{23}(P\otimes 1) = P\otimes d \quad \mbox{in}\quad \C^m \otimes\C^m \otimes \C[\GA].
\end{equation*}
In other words, the embedding $\C\to\C^m\otimes\C^m$, $1\mapsto P$, intertwines the unitary representations~$d$ and $V\otimes V$ of $\GA$.

\begin{lemma}[cf.~{\cite[Lemma~2.3]{Mrozinski}}]
The matrix $|A\bar A|$ is a self-intertwiner of the representation~$V$.
\end{lemma}

\bp
By~\eqref{eq:Vc}, we have
$$
V=V^{cc}=((A^t)^{-1}VA^td^*)^c=d(\bar{A}^t)^{-1}\big((A^t)^{-1}VA^td^*\big)\bar{A}^t=d(A^t\bar A^t)^{-1}VA^t\bar A^td^{-1}.
$$
As $V$ is unitary, it follows that it commutes with
$$
A^t\bar A^td^{-1}(A^t\bar A^td^{-1})^*=|(A^t\bar A^t)^*|^2=|A\bar A|^2,
$$
hence with $|A\bar A|$.
\ep

Therefore we can hope that the representation $V$ of $\GA$ is irreducible only when $|A\bar A|$ is scalar, or equivalently, $A\bar A$ is unitary up to a scalar factor. By \cite[Lemma~1.4]{HN21}, if $m\ge2$, the last condition is satisfied exactly for the Temperley--Lieb polynomials. Furthermore, by results of~\cite{Mrozinski}, in this case $V$ is indeed irreducible and the representation ring of $\GA$ is isomorphic to that of~$U(2)$, with $V$ corresponding to the fundamental representation of $U(2)$ and $d$ to the determinant. More precisely, $\GA$ is monoidally equivalent to $U_q(2)$ for a uniquely defined $0<q\le1$ by the following theorem, which follows from the results in \cite[Section~3]{Mrozinski}.

\begin{theorem}[{\cite{Mrozinski}}] \label{thm:mrozinski}
Let $A \in \GL_m(\C)$ and $C \in \GL_k(\C)$ be matrices such that~$A\bar{A}$ and~$C\bar{C}$ are unitary and $\Tr(A^*A) = \Tr(C^*C)$.
Consider the noncommutative polynomials $P=\sum_{i,j}a_{ij}X_iX_j$ and $Q=\sum_{s,t}c_{st}X_sX_t$.
Then the universal unital algebra $\mathcal{B}(\GC,\GA)$ generated by an invertible element $z$ and elements $y_{ij}$, $1 \leq i \leq k$, $1 \leq j \leq m$, such that
$Y=(y_{ij})_{i,j}$ satisfies
\[ Y A Y^t = zC \quad\text{and}\quad Y^t \bar C Y = z\bar A, \]
is an $\GC$-$\GA$-Galois object, with the commuting comodule structures
\[ \C[\GC]\otimes \mathcal{B}(\GC,\GA) \xleftarrow{\delta_Q} \mathcal{B}(\GC,\GA) \xrightarrow{\delta_P} \B(\GC,\GA)\otimes \C[\GA]\]
determined by
\[(\iota\otimes \delta_P)(Y) = Y_{12}V^P_{13}, \quad \delta_P(z) = z \otimes d^P,\]
and
\[ (\iota\otimes\delta_Q)(Y) = V^Q_{12}Y_{13}, \quad \delta_Q(z) = d^Q\otimes z,\]
where $V^P$, $d^P$ (resp., $V^Q$, $d^Q$) denote the defining representations of $\GA$ (resp., $\GC$).
\end{theorem}

We remark that in the notation of \cite{Mrozinski} the algebra $\mathcal{B}(\GC,\GA)$ is $\mathcal{G}(\bar C,C \, | \, \bar A, A)$.

\smallskip

We introduce a $*$-structure on $\B(\GC,\GA)$ by letting
$$
z^* = z^{-1}\quad \text{and}\quad Y^c = (C^t)^{-1}YA^tz^{-1}.
$$
It is not difficult to check that then $\B(\GC,\GA)$ can be equivalently described as a universal unital $*$-algebra generated by a unitary $z$ and elements $y_{ij}$, $1 \leq i \leq k$, $1 \leq j \leq m$, such that
\[Y = (y_{ij})_{i,j} \ \ \text{is unitary and}\ \ Y = C^tY^c(A^t)^{-1}z .\]
It is also easily verified that $\delta_P$ and $\delta_Q$ are $*$-preserving, so that $\mathcal{B}(\GC,\GA)$ is a $*$-bi-Galois object. From Corollary~\ref{cor:star-criterion} we then get the following result.

\begin{proposition}\label{prop:C-star-structure}
In the setting of Theorem~\ref{thm:mrozinski}, the $*$-algebra $\B(\GC,\GA)$ admits a C$^*$-comple\-tion.
\end{proposition}

We denote by $B(\GC,\GA)$ the C$^*$-envelope of $\B(\GC,\GA)$. Then $B(\GC,\GA)$  is a C$^*$-$\GC$-$\GA$-Galois object. In general, the actions of $\GA$ and $\GC$ on $B(\GC,\GA)$ are not reduced, but they will be reduced in our main cases of interest.

\smallskip

Now, let us fix a matrix $A \in\GL_m(\C)$ ($m\ge2$) such that $A\bar{A}$ is unitary and consider the corresponding Temperley--Lieb polynomial $P=\sum_{i,j}a_{ij}X_iX_j$. By~\cite[Proposition~1.5]{HN21}, there is a unitary $w\in U(m)$ such that
$$
wAw^t=\begin{pmatrix}
0 & & a_1\\
 & \reflectbox{$\ddots$} & \\
 a_m & & 0
\end{pmatrix}.
$$
Therefore by applying a unitary change of variables we may, and will, assume that $P$ has the form $\sum^m_{i=1}a_iX_iX_{m-i+1}$ (with $|a_ia_{m-i+1}|=1$ for all $i$).

Let $0<q\le 1$ be such that
\begin{equation}\label{eq:q}
q+q^{-1}=\sum^m_{i=1}|a_i|^2.
\end{equation}
Applying Theorem~\ref{thm:mrozinski} and Proposition~\ref{prop:C-star-structure} to $Q=q^{-1/2}X_1X_2-q^{1/2}X_2X_1$ and $P$, we get a C$^*$-$U_q(2)$-$\GA$-Galois object $B(U_q(2),\GA)$.

\begin{lemma}\label{lem:reduced}
The action of $\GA$ on $B(U_q(2),\GA)$ is reduced.
\end{lemma}

\bp
The quantum group $U_q(2)$ is coamenable, because it has the same fusion rules and classical dimension function as the compact group $U(2)$ (see \cite[Theorems~2.7.10 and~2.7.12]{NT13}). It follows that the action of $U_q(2)$ on $B(U_q(2),\GA)$ is reduced. In other words, the unique $U_q(2)$-invariant state $\phi$ on $B(U_q(2),\GA)$ is faithful. Since $\phi$ is also the unique $\GA$-invariant state, this means that the action of $\GA$ is reduced.
\ep

Let us now look at the structure of $B(U_q(2),\GA)$ more carefully. By definition, it is a universal C$^*$-algebra generated by elements $y_{ij}$, $1\le i\le 2$, $1\le j\le m$, and a unitary $z$ such that
$$
Y = (y_{ij})_{i,j} \ \ \text{is unitary},\ \ y_{2j}^*=-q^{-1/2}a_jy_{1,m-j+1}z^{-1},\ \ y_{1j}^*=q^{1/2}a_jy_{2,m-j+1}z^{-1}.
$$
It follows that $y_{2j}=-q^{-1/2}\bar a_jzy_{1,m-j+1}^*=-\bar a_j a_{m-j+1}zy_{2j}z^{-1}$, or equivalently,
$$
zy_{2j}z^{-1}=-a_j \bar a_{m-j+1}y_{2j}.
$$
As unitarity of $Y$ is equivalent to that of $\begin{pmatrix}z^* & 0\\ 0 & 1\end{pmatrix}Y$, by letting $y_j=y_{2j}$ we arrive at the following description of $B(U_q(2),\GA)$.

\begin{lemma}\label{lem:generators-relations}
The C$^*$-algebra $B(U_q(2),\GA)$ is a universal unital C$^*$-algebra generated by elements $y_i$, $1\le i\le m$, and a unitary $z$ such that $zy_{i}z^{-1}=-a_i \bar a_{m-i+1}y_{i}$ and
\begin{equation}\label{eq:unitary}
\begin{pmatrix}
q^{1/2}\bar{a}_1y_{m}^* & q^{1/2}\bar{a}_2y_{m-1}^* & \cdots & q^{1/2}\bar{a}_m y_{1}^*\\
y_{1} & y_{2} & \cdots & y_{m}
\end{pmatrix}\quad\text{is unitary}.
\end{equation}
The right action of $\GA$ is given by
$$
\delta(y_i)=\sum^m_{k=1}y_k\otimes v_{ki},\quad\delta(z)=z\otimes d.
$$
\end{lemma}

In order to relate $B(U_q(2),\GA)$ to $\OO_P$, recall that $\HH_P$ is an $\GA$-equivariant subproduct system and as a result, see~\cite[Section~2]{HN21}, the C$^*$-algebras $\TT_P$ and $\OO_P$ carry actions of $\GA$. Namely, the Hilbert space $H_n=f_n(\C^m)^{\otimes n}$ carries a representation $U_n$ of $\GA$ obtained by restriction from $V^{\otimes n}$. The right action of $\GA$ on $\TT_P$ is given by
$$
\delta(S)=U_P(S\otimes1)U_P^*,\quad\text{where}\quad U_P=\bigoplus^\infty_{n=0}U_n.
$$
Note that by the monoidal equivalence $\GA\sim_\otimes U_q(2)$, the quantum group $\GA$ has the same fusion rules as $U(2)$. As a result the representations $U_n\otimes d^l$, $n\ge0$, $l\in\Z$, are irreducible and pairwise nonequivalent, they exhaust the irreducible representations of $\GA$ up to equivalence, and the fusion rules are described by
\begin{equation}\label{eq:fusion-rules}
d\otimes U_n\cong U_n\otimes d,\qquad U_{k} \otimes U_l \cong U_{k+l} \oplus (U_{k+l-2}\otimes d) \oplus \cdots \oplus (U_{|k-l|}\otimes d^{\,\min\{k,l\}}).
\end{equation}
But we will not need this until the next section.

Consider also the unitary
$$
u=-A\bar A=-\begin{pmatrix}a_1\bar a_m & &0\\
& \ddots & \\
0 & & a_m\bar a_1\end{pmatrix}.
$$
It leaves $P$ invariant, hence the conjugation by the restriction of $\bigoplus_{n\ge0}u^{\otimes n}$ to $\F_P$ defines an automorphism $\beta$ of $\TT_P$ such that
$$
\beta(S_i)=-a_i\bar a_{m-i+1}S_i.
$$
We denote by $\bar\beta$ the corresponding automorphism of $\OO_P$. Let $s_i$ be the image of $S_i$ in $\OO_P$.

\begin{theorem} \label{thm:CP-crossed}
Assume $P=\sum^m_{i=1}a_iX_iX_{m-i+1}$ is a noncommutative polynomial ($m\ge2$) with $|a_ia_{m-i+1}|=1$ for all $i$, and let $q\in(0,1]$ be given by~\eqref{eq:q}.
Then there is a C$^*$-algebra isomorphism $B(U_q(2),\GA)\cong \mathcal{O}_P\rtimes_{\bar\beta} \Z $ that maps $y_i$ into $s_i$ and $z$ into the canonical unitary implementing $\bar\beta$.
\end{theorem}

\begin{proof}
Let $B$ be the C$^*$-subalgebra of $B(U_q(2),\GA)$ generated by the elements $y_i$. Since the action of $\GA$ on $B(U_q(2),\GA)$ is ergodic and reduced by Lemma~\ref{lem:reduced}, the same is true for the action on $B$. It follows that $B$ does not have any nontrivial equivariant quotients.

It is clear from Lemma~\ref{lem:generators-relations} that condition~\eqref{eq:unitary} gives a complete set of relations in~$B$.
By \cite[Proposition 1.7]{HN21} we already know that these relations are satisfied by the elements $s_i\in\OP$. We therefore get a surjective $*$-homomorphism $\psi\colon B\to\OO_P$ such that $\psi(y_i)=s_i$. This homomorphism is $\GA$-equivariant, hence it is an isomorphism.

By Lemma~\ref{lem:generators-relations} we have $B(U_q(2),\GA)\cong B\rtimes_{\Ad z}\Z$. This allows us to extend $\psi$ to an isomorphism $B(U_q(2),\GA)\cong\OO_P\rtimes_{\bar\beta}\Z$.
\end{proof}

\begin{corollary}
The C$^*$-algebras $\TT_P$ and $\OO_P$ are nuclear.
\end{corollary}

\bp
As $U_q(2)$ is coamenable and acts ergodically on $B(U_q(2),\GA)$, the C$^*$-algebra $\OO_P\rtimes_{\bar\beta}\Z\cong B(U_q(2),\GA)$ is nuclear by~\cite{DLRS}. This implies the result.
\ep

\begin{remark} \label{rem:connected}
The most nontrivial part of the proof of Theorem~\ref{thm:mrozinski} given in~\cite{Mrozinski} is that $\B(U_q(2),\GA)\ne0$. The proof of Theorem~\ref{thm:CP-crossed}, showing that a quotient of $\B(U_q(2),\GA)$ is a dense subalgebra of $\OO_P\rtimes_{\bar\beta}\Z$, provides another argument for this. In other words, subproduct systems can be used to give an alternative proof of Theorem~\ref{thm:mrozinski} and simultaneously show that the $*$-algebras $\B(\GC,\GA)$ have nonzero representations on Hilbert spaces. Note that in order to prove Proposition~\ref{prop:C-star-structure} we then no longer need Corollary~\ref{cor:star-criterion}, it suffices to apply implication $(2)\Rightarrow(1)$ in Theorem~\ref{thm:Bichon}. This implication is, in turn, an easy consequence of Lemma~\ref{lem:Bichon}. Note also that thanks to~\cite[Lemma~1.6]{HN21} we do know that $\OO_P\ne0$. \ee
\end{remark}

From the proof of Theorem~\ref{thm:CP-crossed} we get a complete set of relations for $\OO_P\cong B$. This allows us to extend \cite[Theorem 3.4]{HN21} to all Temperley--Lieb polynomials. The proof is identical to that in \cite{HN21} and is therefore omitted, but we formulate the result for general Temperley--Lieb polynomials up to normalization.

\begin{theorem}\label{thm:relations}
Assume $A=(a_{ij})_{i,j}\in\GL_m(\C)$ ($m\ge2$) is such that $A\bar A$ is unitary. Let $q\in(0,1]$ be the number such that $\Tr(A^*A)=q+q^{-1}$. Consider the noncommutative polynomial $P=\sum^m_{i,j=1}a_{ij}X_iX_j$.  Then $\mathcal{T}_P$ is a universal C$^*$-algebra generated by $c = C(\Z_+ \cup \{\infty\})$ and $S_1,S_2,...,S_m$ satisfying the relations
\[ fS_i = S_i\gamma(f) \quad (f \in c,\ 1 \leq i \leq m), \quad \sum^m_{i=1} S_iS_i^* = 1 - e_0, \quad \sum^m_{i,j=1} a_{ij}S_iS_j = 0,  \]
\[S_i^*S_j + \phi\sum^m_{k,l=1}a_{ik}\bar{a}_{jl}S_kS_l^* = \delta_{ij}1 \quad (1 \leq i,j \leq m). \]
\end{theorem}

Here $c$ is identified with a unital subalgebra of $\K(\mathcal{F}_P)+\C1 \subset \mathcal{T}_P$, with the characteristic function $e_n$ of $\{n\}$ identified with the projection $\mathcal{F}_P \to H_n$. The endomorphism $\gamma\colon c\to c$ is the shift to the left, and $\phi\in c$ is defined by
\begin{equation} \label{eq:phi(n)}
\phi(n)= \dfrac{[n]_q}{[n+1]_q},\quad\text{where}\quad [n]_q= \dfrac{q^{n}-q^{-n}}{q - q^{-1}}.
\end{equation}

\bigskip

\section{K-theory of \texorpdfstring{$\TP$}{TP}}\label{sec:K-theory-T}

The goal of this section is to prove the following result generalizing and strengthening \cite[Theorem~5.1]{HN21}.

\begin{theorem} \label{thm:KK-Toeplitz}
For every Temperley--Lieb polynomial $P$, the embedding map $\C\to\TT_P$ is a $KK^{\GA}$-equivalence.
\end{theorem}

The strategy is similar to that in~\cite{HN21}. Using the monoidal equivalence $\GA\sim_\otimes U_q(2)$ and arguing as in~\cite[Section~5]{HN21}, it suffices to prove the theorem for the polynomials
$$
P=q^{-1/2}X_1X_2-q^{1/2}X_2X_1,\quad 0<q\le1.
$$
Fix such a polynomial $P$. We know by \cite[Theorem~5.1]{HN21} that $\C\to\TT_P$ is a $KK^{SU_q(2)}$-equivalence, and we need to upgrade this to a $KK^{U_q(2)}$-equivalence.

For the proof of the $KK^{SU_q(2)}$-equivalence in \cite{HN21} we used the Baum--Connes conjecture for the discrete dual of $SU_q(2)$ established by Voigt~\cite{Voigt-BC}. An inspection of Voigt's arguments shows that with minor modifications they work for the dual of $U_q(2)$ as well; we give the precise formulation of the result and some details of the necessary modifications in Appendix~\ref{app:A}. A~standard consequence of this, see~\cite[Proposition~5.2]{HN21}, is that given separable $U_q(2)$-C$^*$-algebras $C$ and $D$, an element of $KK^{U_q(2)}(C,D)$ is a $KK^{U_q(2)}$-equivalence if and only if it defines a $KK$-equivalence between $C\rtimes U_q(2)$ and $D \rtimes U_q(2)$.

In order to take advantage of this consequence of the Baum--Connes conjecture, consider the exact sequence
\begin{equation}
\label{eq:SES_crossed}
0 \to \K(\mathcal{F}_P) \rtimes U_q(2)  \to \mathcal{T}_P \rtimes U_q(2) \to \mathcal{O}_P \rtimes U_q(2) \to 0.
\end{equation}
Recall that given an action $\alpha\colon B \to B \otimes C(U_q(2))$, the crossed product is by definition the C$^*$-algebra
\[B \rtimes U_q(2) = \overline{\alpha(B)(1\otimes \rho(C^*(U_q(2))))}^{\| \cdot \|},\]
where $\rho\colon C^*(U_q(2))\to B(L^2(U_q(2)))$ is the right regular representation.

The C$^*$-algebras $\K(\mathcal{F}_P) \rtimes U_q(2)$ and $\mathcal{O}_P \rtimes U_q(2)$ in~\eqref{eq:SES_crossed} have a simple description.
First, by Theorem~\ref{thm:CP-crossed} we have
\[\OP \cong {}^\T B(U_q(2),U_q(2))={}^\T C(U_q(2))=C(\T \backslash U_q(2)) ,\]
where $\mathbb{T}$ denotes the circle $\begin{pmatrix}\T & 0\\ 0 & 1\end{pmatrix}\subset U_q(2)$ formally defined by the surjective map
\[C(U_q(2)) \to C(\T), \quad v_{11},d\mapsto z,\quad v_{22}\mapsto1,\quad  v_{12},v_{21}\mapsto0.\]
Since $C(U_q(2))\rtimes U_q(2)\cong\K(L^2(U_q(2)))$ by the Takesaki--Takai duality~\cite[Theorem~5.33]{DC}, we obtain
\[\mathcal{O}_P\rtimes U_q(2) \cong {}^\T \K(L^2(U_q(2))) = c_0\text{-}\bigoplus_{s\in \Z} \K(\mathcal{H}_s), \]
where $L^2(U_q(2)) = \bigoplus_{s\in\Z}\mathcal{H}_s$ is decomposed according to the grading induced by the left action of $\T$ and $c_0\text{-}\bigoplus_{s\in \Z} \K(\mathcal{H}_s)$ consists by definition of families of operators $(T_s)_{s\in\Z}$ such that $T_s\in\K(\mathcal{H}_s)$ and $\|T_s\|\to0$ as $|s|\to+\infty$.

Next, the conjugation by $U_P^*$ gives an isomorphism
\[\K(\mathcal{F}_P) \rtimes U_q(2) \cong \K(\mathcal{F}_P) \otimes \rho(C^*(U_q(2))) \cong c_0\text{-}\bigoplus_{k\ge0,\ l\in\Z} \K(\mathcal{F}_P \otimes H_{k,l} ) , \]
where $H_{k,l} = H_k$ is the Hilbert underlying the representation $U_{k,l} := U_k \otimes d^l$.

Now, we may rewrite the short exact sequence (\ref{eq:SES_crossed}) as
\begin{equation} \label{eq:SES-rewritten}
 0 \to c_0\text{-}\bigoplus_{k,l} \K(\FP \otimes H_{k,l}) \to \TP \rtimes U_q(2) \to c_0\text{-}\bigoplus_{s} \K(\mathcal{H}_s) \to 0.
\end{equation}
From this we conclude in particular that $\TP \rtimes U_q(2)$ is a Type I C$^*$-algebra with trivial $K_1$-group. Because the same holds for $C^*(U_q(2))$, the Universal Coefficient Theorem applies and we see that in order to prove Theorem \ref{thm:KK-Toeplitz} it suffices to show that the map
\[ \pi\colon C^*(U_q(2)) \to \TP \rtimes U_q(2), \quad a \mapsto 1 \otimes \rho(a),  \]
induces an isomorphism of the $K_0$-groups.

Before we embark on the computation of $\pi_*\colon K_0(C^*(U_q(2)))\to K_0(\TT_P\rtimes U_q(2))$, we introduce some notation and record a small lemma. We identify $K_0(C^*(U_q(2))$ with the representation ring $\bigoplus_{k,l} \Z [U_{k,l}]$ of $U_q(2)$, so $[U_{k,l}]$ corresponds to a rank one projection in $B(H_{k,l})\subset C^*(U_q(2))$. We also fix rank one projections $p_{k,l} \in \K(\FP\otimes H_{k,l})$ and $q_s \in \K(\mathcal{H}_s)$ for each $k \in \Z_+,$ $l,s \in \Z$.

\begin{lemma}\label{lem:rho-star}
With the notation as above,
$$
\rho_*( [U_{k,l}]) = \sum_{s=0}^k [q_{-l-s}]\quad\text{for all}\quad k \in \Z_+,\ l \in \Z,
$$
where $\rho$ is the right regular representation viewed as a homomorphism
$$
C^*(U_q(2))\to c_0\text{-}\bigoplus_{s\in\Z}\K(\HH_s)\subset\K(L^2(U_q(2))).
$$
\end{lemma}

\begin{proof}
By the fusion rules~\eqref{eq:fusion-rules}, the dual of the $U_q(2)$-module $H_{k,l}$ is $H_{k,-l-k}$. By the Peter--Weyl theory, it follows that the space $L^2(U_q(2))$ decomposes as
$$
L^2(U_q(2))\cong \bigoplus_{k,l}H_{k,-l-k}\otimes H_{k,l}
$$
with respect to the left and right regular representations acting on the first and second factor, resp.
It follows that $\rho_*([U_{k,l}]) = \sum_s  c_s[q_s]$, where $c_s$ is the multiplicity of the weight $s$ in $H_{k,-l-k}$ with respect to our chosen torus $\T\subset U_q(2)$. These multiplicities are the same as for~$U(2)$, where~$d$ is the determinant and the dual $U_{k,-k}$ of $U_k$ is equivalent to the representation by the change of variables on the space of homogeneous polynomials (in two commuting variables) of degree $k$. We then see that $H_{k,-l-k}$ has weights ${-l},{-l-1},\dots,{-l-k}$, each of multiplicity one. (See also Lemma~\ref{lem:weights} for a similar more formal computation.) This gives the result.
\end{proof}

\begin{proof}[Proof of Theorem \ref{thm:KK-Toeplitz}]
Let $[\tilde q_s]=\pi_*([U_{0,-s}])$. Then, by the previous lemma, $[\tilde q_s]$ is a lift of $[q_s]$ to $K_0(\TT_P\rtimes U_q(2))$. From~\eqref{eq:SES-rewritten} it follows that $K_0(\TT_P\rtimes U_q(2))$ is a free abelian group with a basis consisting of the classes $[p_{m,n}]$ and $[\tilde q_s]$ ($m\ge0$, $n,s\in\Z$). Therefore, for fixed $k$ and $l$, we have
\[\pi_*([U_{k,l})] = \sum_{m,n} c_{m,n} [p_{m,n}] + \sum_s c_s [\tilde{q}_s] \quad\text{for some}\quad c_s, c_{m,n} \in \Z. \]

Consider the commutative diagram
\begin{equation}\label{eq:diagram}
\begin{tikzcd}
\K(\mathcal{F}_P \otimes H_{m,n}) & \arrow{l}[swap]{} \mathcal{T}_P \rtimes U_q(2) \arrow[r]  & c_0\text{-}\bigoplus_{s \in \Z} \K(\mathcal{H}_s) \\
& C^*(U_q(2)) \arrow{u}{\pi} \arrow{ur}[swap]{\rho} \arrow{ul}{ \pi_{U_P\otimes U_{m,n}}  } &
\end{tikzcd}
\end{equation}
where $\pi_{U_P\otimes U_{m,n}}$ is the integrated form of the representation $U_P\otimes U_{m,n}$ of $U_q(2)$ and the left horizontal arrow comes from the action of $\TT_P\rtimes U_q(2)$ on its ideal $\K(\F_P)\rtimes U_q(2)\cong c_0\text{-}\bigoplus_{m,n} \K(\FP \otimes H_{m,n})$. Note that a priori we have only representations of $C^*(U_q(2))$ and $\TT_P\rtimes U_q(2)$ on $\F_P\otimes H_{m,n}$, but these are representations by compact operators since the multiplicity of $U_{k,l}$ in $U_P\otimes U_{m,n}$ is finite, see below.

The right triangle in~\eqref{eq:diagram} combined with (\ref{eq:SES-rewritten}) and Lemma~\ref{lem:rho-star} yields $c_s = 1$ for $s=-l,-l-1,...,-l-k$ and $c_s = 0$ otherwise. Then, looking at the left triangle in~\eqref{eq:diagram}, we get
\[ (\pi_{U_P\otimes U_{m,n}})_*([U_{k,l})] = c_{m,n}[p_{m,n}] + \sum_{s=0}^{k}(\pi_{U_P\otimes U_{m,n}})_*([U_{0,l+s}])
\]
in $K_0(\K(\F_P\otimes H_{m,n}))=\Z$. The number $(\pi_{U_P\otimes U_{m,n}})_*([U_{k,l})]$ is the multiplicity of $U_{k,l}$ in $U_P\otimes U_{m,n}$, equivalently, by Frobenius reciprocity (see \cite[Theorem~2.2.6]{NT13}), it is
$$
\dim\Mor_{U_q(2)}(U_{k,l}\otimes U_{m,-m-n},U_P).
$$
By the fusion rules~\eqref{eq:fusion-rules},
$$
U_{k,l}\otimes U_{m,-m-n}\cong U_{k+m,l-m-n}\oplus U_{k+m-1,l-m-n+1}\oplus\dots\oplus U_{|k-m|,l-m-n+\min\{m,k\}}.
$$
Therefore the space $\Mor_{U_q(2)}(U_{k,l}\otimes U_{m,-m-n},U_P)$ is one-dimensional if the interval $$[l-m-n,l-m-n+\min\{m,k\}]$$ contains zero, and it is zero otherwise. In other words,
$$
(\pi_{U_P\otimes U_{m,n}})_*([U_{k,l})]=\begin{cases}
                                         1, & \mbox{if } m+n\in[l,l+\min\{m,k\}], \\
                                         0, & \mbox{otherwise}.
                                       \end{cases}
$$
For the same reason
$$
(\pi_{U_P\otimes U_{m,n}})_*([U_{0,l+s}])=\begin{cases}
                                         1, & \mbox{if } m+n=l+s, \\
                                         0, & \mbox{otherwise}.
                                       \end{cases}
$$
It follows that
\begin{align*}
c_{m,n}&=(\pi_{U_P\otimes U_{m,n}})_*([U_{k,l})]-\sum^k_{s=0}(\pi_{U_P\otimes U_{m,n}})_*([U_{0,l+s}])\\
&=\begin{cases}
                                         -1, & \mbox{if } l+\min\{m,k\}<m+n\le l+k, \\
                                         0, & \mbox{otherwise}.
                                       \end{cases}
\end{align*}

In conclusion,
$$
\pi_*([U_{k,l})]=-\sum^{k-1}_{m=0}\sum^{l+k-m}_{n=l+1}[p_{m,n}]+\sum^{-l}_{s=-l-k}[\tilde q_s].
$$
The elements $\pi_*([U_{0,l}])=[\tilde q_{-l}]$, $l\in\Z$, are independent. Observe next that in the formula for $\pi_*([U_{k,l})]$ for $k\ge1$ the combination of $[p_{m,n}]$'s with the largest $m$ is simply $-[p_{k-1,l+1}]$. From this it is easy to see by induction on~$k_0$ that the elements $\pi_*([U_{k,l})]$, with $0\le k\le k_0$ and $l\in\Z$, are independent and generate the same subgroup as $[p_{m,n}]$ and $[\tilde q_s]$, with $0\le m<k_0$ and $s\in\Z$. Hence $\pi_*$ is an isomorphism.
\end{proof}

\bigskip

\section{K-theory of \texorpdfstring{$\OP$}{OP}}\label{sec:K-theory-O}

In this section we compute the $K$-theory of $\OO_P$ for all Temperley--Lieb polynomials $P=\sum^m_{i,j=1}a_{ij}X_iX_j$. For this we first construct an inverse of the embedding map $\C\to\TT_P$  in~$KK^{\GA}$. The construction is an adaption of the work of Arici and Kaad~\cite{Arici-Kaad}, who defined an inverse in $KK^{SU(2)}$ for the polynomials $\sum^m_{i=1}(-1)^iX_iX_{m-i+1}$.

\smallskip

Let us fix $P$. By rescaling we may assume that $A\bar A$ is unitary. Write $\F$ for $\F_P$. Denote by $H_{k,l}=H_k$ the underlying space of the representation $U_k\otimes d^l$, as in the previous section, and~put
$$
\F[1]=\bigoplus^\infty_{n=0}H_{n,1}.
$$
We are going to define an $\GA$-equivariant quasi-homomorphism $(\phi_+,\phi_-)\colon\TT_P\to\K(\F\oplus\F[1])$.

Using that $\F[1]=\F$ as Hilbert spaces, we define $\phi_+\colon \TT_P\to B(\F\oplus\F[1])$ by
$$
\phi_+(S)=(S,S)\quad (S\in\TT_P).
$$
The representations $U_P$ and $U_P\otimes d$ define the same right actions on the bounded operators on~$\F$, so $\phi_+$ is indeed $\GA$-equivariant.

\smallskip

The homomorphism $\phi_-$ will be defined using an identification of $\F\otimes H_{1,0}$ with a codimension one subspace of $\F\oplus\F[1]$. We need some preparation to describe this identification.

Recall that~$e$ denotes the projection onto $\C P\subset\C^m\otimes\C^m$, and $f_n$ the projection $(\C^m)^{\otimes n}\to H_n=H_{n,0}$, so
\begin{equation*}
\label{eq:Jones-Wenzl-projections}
f_0 = 1_\C, \quad f_1 = 1_{\C^m}, \quad f_n = 1 - \bigvee_{i=0}^{n-2} 1^{\otimes i}\otimes e \otimes 1^{\otimes (n-i-2)}\ \ \text{for}\ \ n\ge2.
\end{equation*}
These are the Jones--Wenzl projections in Temperley--Lieb algebras and a key known fact, which we already relied on in~\cite{HN21}, is that they satisfy the recurrence relation
\begin{equation*}\label{eq:Wenzl-recurrence}
f_{n+1} = f_n \otimes 1 -  [2]_q\phi(n)(f_{n}\otimes1)(1^{\otimes(n-1)}\otimes e)(f_{n}\otimes1),
\end{equation*}
where $0<q\le1$ is such that $q+q^{-1}=\Tr(A^*A)$ and $\phi(n)$ is given by~\eqref{eq:phi(n)}, so $\phi(n)=\dfrac{[n]_q}{[n+1]_q}$. Let $v\colon \C\to \C^m\otimes\C^m$ be the isometry $1\mapsto [2]_q^{-1/2}P$, so that $vv^*=e$. Define an $\GA$-equivariant partial isometry
\begin{equation}\label{eq:Wenzl-isometry}
w_{n}:=([2]_q\phi(n+1))^{1/2}(f_{n+1}\otimes1)(1^{\otimes n}\otimes v)\colon H_{n,1}\to H_{n+1,0}\otimes H_{1,0}.
\end{equation}
This is known to be an isometry, see, e.g., the proof of~\cite[Lemma~2.5.8]{NT13}. Denote by $v_n$ ($n\ge1$) the embedding $H_{n,0}\to H_{n-1,0}\otimes H_{1,0}$. Therefore $v_1=1$ and, by the fusion rules~\eqref{eq:fusion-rules}, we have $w_{n-1}w_{n-1}^*+v_{n+1}v_{n+1}^*=1$ for all $n\ge1$.

Write $\mathcal{F}_+ = \bigoplus_{n\geq 1} H_{n,0}$. Then we obtain $\GA$-equivariant isometries
\[\Theta = \bigoplus^\infty_{n=1} v_{n} \colon \mathcal{F}_+ \to \mathcal{F}\otimes H_{1,0}, \quad \Phi = \bigoplus^\infty_{n=0} w_n \colon \mathcal{F}[1] \to \mathcal{F}\otimes H_{1,0}\]
such that their ranges are orthogonal and their sum is $\F\otimes H_{1,0}$. We extend $\Theta$ to a partial isometry $\F\to \F\otimes H_{1,0}$ by letting $\Theta=0$ on $H_{0,0}$. Then $\Theta+\Phi$ is an equivariant coisometry $\F\oplus\F[1]\to \F\otimes H_{1,0}$ with one-dimensional kernel $H_{0,0}\oplus0$. Define $\phi_-\colon\TT_P\to B(\F\oplus\F[1])$ by
$$
\phi_-(S):=(\Theta+\Phi)^*(S\otimes1)(\Theta+\Phi)\qquad (S\in\TT_P),
$$
or in matrix form, by
$$
\phi_-(S):=\begin{pmatrix}
            \Theta^*(S\otimes1)\Theta &  \Theta^*(S\otimes1)\Phi\\
            \Phi^*(S\otimes1)\Theta & \Phi^*(S\otimes1)\Phi
          \end{pmatrix}.
$$

\begin{lemma}[cf.~{\cite[Proposition~5.4]{Arici-Kaad}}]
For all $i=1,\dots,m$, the following holds:
\begin{itemize}
\item[(i)] $(S_i^* \otimes 1)\Theta = \Theta S_i^*$;
\item[(ii)] the operator $(S_i^*\otimes 1)\Phi - \Phi S_i^*$ is compact.
\end{itemize}
\end{lemma}
\begin{proof}
Let $(\xi_i)_i$ be the standard basis in $H_1=\C^m$. For every $i$, denote by $T_i$ the creation operator of tensoring on the left by $\xi_i$ on the full Fock space $\bigoplus_{n\ge0}H_1^{\otimes n}$. Then, by the definition of subproduct systems, the subspace $\F$ is invariant under $T_i^*$ and $S_i^*=T_i^*|_\F$. Part (i) follows directly from this.

To prove part (ii), let $\zeta \in H_{n,1}=H_{n}=f_{n}H_1^{\otimes n}$. Using again that $S_i^*=T_i^*|_\F$, we get
\[(S_i^*\otimes 1)\Phi\zeta - \Phi S_i^*\zeta = T_i^*(w_n-1\otimes w_{n-1})\zeta.\]
It is therefore enough to check that the norms of the operators
$$
w_n-(1\otimes w_{n-1})f_n\colon H_1^{\otimes n}\to H_1^{\otimes(n+2)}
$$
converge to~$0$ as~$n$ tends to infinity, where we view $w_n$ as a partial isometry $H_1^{\otimes n}\to H_1^{\otimes(n+2)}$ with the source projection $f_{n}$. Note that expression~\eqref{eq:Wenzl-isometry} for $w_n$ does not change, since $f_{n+1}\le f_{n}\otimes1$.

As $f_n\le1\otimes f_{n-1}$, we have
\begin{multline*}
(w_n-(1\otimes w_{n-1})f_n)^*(w_n-(1\otimes w_{n-1})f_n) \\
 = f_n-w_n^*(1\otimes w_{n-1})f_n-f_n(1\otimes w_{n-1}^*)w_n+f_n.
\end{multline*}
From \eqref{eq:Wenzl-isometry}, using again that $f_{n+1}\le1\otimes f_n$, we get
$$
w_n^*(1\otimes w_{n-1})=\dfrac{\phi(n)^{1/2}}{\phi(n+1)^{1/2}}w_n^*w_n=\dfrac{\phi(n)^{1/2}}{\phi(n+1)^{1/2}}f_n.
$$
Therefore
$$
\|w_n-(1\otimes w_{n-1})f_n\|=\Big(2-2\dfrac{\phi(n)^{1/2}}{\phi(n+1)^{1/2}}\Big)^{1/2}
=\sqrt{2}\big(1-(1 - [n+1]_q^{-2})^{1/2}\big)^{1/2}.
$$
This expression obviously converges to $0$ as $n\to\infty$.
\end{proof}

From this we immediately get the following.

\begin{corollary}
For every $S\in\TT_P$, the operator $\phi_+(S)-\phi_-(S)$ is compact.
\end{corollary}

Thus the $*$-homomorphisms $\phi_+, \phi_- \colon \TP \to B(\mathcal{F}\oplus \mathcal{F}[1])$ indeed define a quasi-homomorphism $\TT_P\to\K(\F\oplus\F[1])$ and therefore an element $[(\phi_+,\phi_-)]\in KK^{\GA}(\TT_P,\C)$.

\begin{theorem} \label{thm:inverse}
For every Temperley--Lieb polynomial $P$, the class $[(\phi_+,\phi_-)]\in KK^{\GA}(\TT_P,\C)$ is the inverse of the embedding $i\colon\C\to\TT_P$ in $KK^{\GA}$.
\end{theorem}

\begin{proof}
Since we already know that $[i]$ is invertible by Theorem~\ref{thm:KK-Toeplitz}, it suffices to show that
$$
[i] \hat{\otimes}_{\TT_P} [(\phi_-, \phi_+)]\in KK^{\GA}(\C,\C)
$$
is the identity. We have $[i] \hat{\otimes}_{\TT_P} [(\phi_+, \phi_-)] = [(\phi_+ \circ i, \phi_- \circ i)]$, and moreover
\[ \phi_+(1) - \phi_-(1) = (e_0, 0), \]
where $e_0$ is the projection $\F\to H_{0,0}$. Since $H_{0,0}$ is the underlying space of the trivial representation of $\GA$, this is exactly what we want.
\end{proof}

Now we stop worrying about equivariance and compute the $K$-theory of $\OP$. The argument is essentially identical to that in \cite{Arici-Kaad} for the polynomials $\sum^m_{i=1}(-1)^iX_iX_{m-i+1}$.

\begin{corollary}\label{cor:K-theory-O}
For every Temperley--Lieb polynomial $P$ in $m$ variables, we have
$$
K_0(\OO_P)\cong\Z/(m-2)\Z,\qquad K_1(\OO_P)\cong\begin{cases} \Z,&\text{if } m=2,\\ 0,&\text{if }m\ge3.\end{cases}
$$
\end{corollary}

\bp
We apply the six-term exact sequence in $K$-theory to the short exact sequence
$$
0 \to \K(\mathcal{F}) \to \TP \to \OP \to 0.
$$
We identify $K_0(\K(\F))$ with $\Z$ and, using the $KK$-equivalence $\C \cong_{KK} \TP$, we also identify $K_0(\TT_P)$ with $\Z$. Then it suffices to check that the homomorphism $\Z\to\Z$ of the $K_0$-groups induced by the embedding $\K(\F)\to\TT_P$ is the multiplication by $2-m$. The composition of this embedding with $[(\phi_+,\phi_-)]\in KK(\TT_P,\C)$ is represented by the quasi-homomorphism $(\psi_+,\psi_-)=(\phi_+,\phi_-)|_{\K(\F)}\colon \K(\F)\to\K(\F\oplus\F[1])$. Clearly, $\psi_\pm$ are well-defined $*$-homomorphisms $\K(\F)\to\K(\F\oplus\F[1])$, hence
$$
[(\psi_+,\psi_-)]=[\psi_+]-[\psi_-]\in KK(\K(\F),\C).
$$
By construction, for any rank one projection $p\in\K(\F)$, the projections $\psi_+(p)$ and $\psi_-(p)$ have ranks $2$ and $m$, resp. This gives us what we need.
\ep

\begin{remark}
Recall that $\GA$ is monoidally equivalent to $U_q(2)$ for a uniquely defined $0<q\le1$ and then, by Theorem~\ref{thm:CP-crossed}, we have $B(U_q(2),\GA)\cong\OO_P\rtimes_{\bar\beta}\Z$. It is easy to see that the automorphism $\beta$ of $\TT_P$ is homotopic to the identity. It follows that $K_i(B(U_q(2),\GA))\cong K_0(\OO_P)\oplus K_1(\OO_P)$, $i=0,1$. Therefore the above corollary also gives us a computation of the $K$-theory of the C$^*$-bi-Galois objects $B(U_q(2),\GA)$. In principle, the Baum--Connes conjecture for the dual of $U_q(2)$ alone should be enough to compute this $K$-theory and then the $K$-theory of~$\OO_P$. But it seems the approach based on the construction of the quasi-homomorphism $(\phi_+,\phi_-)$ leads faster to the result.
\end{remark}

\bigskip

\appendix

\section{The Baum--Connes conjecture for the dual of \texorpdfstring{$U_q(2)$}{Uq2}}\label{app:A}

Fix a number $q \in (0,1)$. The quantum group $U_q(2)$ coincides with $\GA$ for the polynomial $P=q^{-1/2}X_1X_2-q^{1/2}X_2X_1$. Therefore, by definition, the $*$-algebra $\C[U_q(2)]$ is generated by the matrix coefficients of a unitary $2\times 2$ matrix $V$ and a unitary element $d$ such that
$$
V\begin{pmatrix}
   0 & q^{-1/2} \\
   -q^{1/2} & 0
 \end{pmatrix}V^t
=\begin{pmatrix}
   0 & q^{-1/2}d \\
   -q^{1/2}d & 0
 \end{pmatrix}.
$$
Let us introduce the following elements:
$$
\alpha=v_{22}^*,\quad \gamma=v_{21}d^*,\quad x=d^*.
$$
Then $\C[U_q(2)]$ can be equivalently described as a universal $*$-algebra with generators $\alpha,\gamma,x$ such that
$$
U=\begin{pmatrix}
   \alpha & -q\gamma^* \\
   \gamma & \alpha^*
 \end{pmatrix}
$$
is unitary and $x$ is unitary, commuting with $\alpha$ and $\gamma$. The comultiplication is determined by
\[\Delta(\alpha) = \alpha\otimes \alpha -q\gamma^*x\otimes \gamma, \quad \Delta(\gamma) = \alpha^*x\otimes \gamma + \gamma\otimes \alpha,\quad \Delta(x) = x \otimes x. \]
The defining representation $V$ of $U_q(2)$ is given by
\begin{equation}\label{eq:UvsV}
V=U\begin{pmatrix}
x^* & 0 \\
0 & 1
\end{pmatrix}
=\begin{pmatrix}
   \alpha x^* & -q\gamma^* \\
   \gamma x^* & \alpha^*
 \end{pmatrix}.
\end{equation}

We see that as a $*$-algebra $\C[U_q(2)]$ decomposes into a tensor product of the $*$-subalgebra generated by $x$ and the $*$-subalgebra generated by $\alpha$ and $\gamma$. The latter subalgebra can be identified with $\C[SU_q(2)]$ with its standard generators $\alpha$ and $\gamma$, and we denote by
$$
i\colon \C[SU_q(2)]\to\C[U_q(2)]
$$
the obvious embedding, which we will use only when we need a clear distinction between elements of these two algebras. Note that this embedding does not respect the coproducts. As a coalgebra, $\C[U_q(2)]$ is a smash coproduct of $\C[\T]$ and $\C[SU_q(2)]$, so that
$\C[U_q(2)]$ is a Hopf algebra with projection \cite{Radford}, see also \cite{KMRS}. Denote by
$$
\pi\colon\C[U_q(2)]\to\C[SU_q(2)]
$$
the projection (or restriction) map, determined by $\pi\circ i=\id$ and $\pi(x)=1$.

We have a maximal $2$-torus $T\subset U_q(2)$, which we decompose as $T=\T\times U(1)$,
$$
T=\{\begin{pmatrix}
  w^{-1}z & 0 \\
  0 & z^{-1}
\end{pmatrix}:\, w\in\T,\ z\in U(1)\}.
$$
Formally, the quantum subgroup
$\T\subset U_q(2)$ is defined by the surjective map
\[ \C[U_q(2)] \to \C[\mathbb{T}] = \C[w,w^{-1}], \quad x \mapsto w, \quad \alpha \mapsto 1, \quad \gamma \mapsto 0 .\]
The quantum subgroup $U(1)$ is the maximal torus of $SU_q(2)$ determined by the map
\[\C[SU_q(2)] \to \C[U(1)] = \C[z,z^{-1}], \quad  \alpha \mapsto z, \quad \gamma \mapsto 0. \]
We identify $\C[T]$ with $\C[z,z^{-1},w,w^{-1}]$. To summarize, we have the following diagram of quantum subgroups:
\begin{center}
\begin{tikzcd}[remember picture]
    U(1) & SU_q(2) & & \C[z,z^{-1}] & \arrow{l}[swap]{\pi_{U(1)}} \C[SU_q(2)]\\
    T  & U_q(2) & & \C[z,z^{-1},w,w^{-1}] \arrow{u} & \C[U_q(2)] \arrow{l}{\pi_T} \arrow{u}[swap]{\pi}\\
\end{tikzcd}
\begin{tikzpicture}[overlay,remember picture]
\path (\tikzcdmatrixname-1-1) to node[midway]{$\subset$}
(\tikzcdmatrixname-1-2);
\path (\tikzcdmatrixname-1-1) to node[midway,sloped]{$\subset$}
(\tikzcdmatrixname-2-1);
\path (\tikzcdmatrixname-1-2) to node[midway,sloped]{$\subset$}
(\tikzcdmatrixname-2-2);
\path (\tikzcdmatrixname-2-1) to node[midway]{$\subset$}
(\tikzcdmatrixname-2-2);
\end{tikzpicture}
\end{center}
\vspace{-20pt}

Under our identification of $\C[SU_q(2)]$ as a $*$-subalgebra of $\C[U_q(2)]$, we have
\[ \C[SU_q(2)/U(1)] = \C[U_q(2)/T], \] giving two pictures of the Podl\'{e}s sphere $S^2_q$. Our main goal is to understand the $D(U_q(2))$-equivariant $KK$-theory of $C(U_q(2)/T)$. We will do this by adapting Voigt's arguments~\cite{Voigt-BC} for $D(SU_q(2))$. Note that in order to make the comparison with~\cite{Voigt-BC} more transparent we work with left actions.

\smallskip

For $l \in \frac{1}{2}\Z_+$, let $U_l$ be the spin $l$ irreducible representation of $SU_q(2)$ with the standard orthonormal basis $\{\xi^l_l, \xi^l_{l-1}, ..., \xi_{-l}^l\}$. This representation has a unique, up to a scalar factor, nonzero morphism into $U^{\otimes 2l}$, so its underlying space can be identified with a subspace of~$(\C^2)^{\otimes 2l}$. This subspace (which is exactly the subspace $f_{2l}(\C^2)^{\otimes 2l}$ used in the main text) is invariant under $U_q(2)$ and determines an irreducible subrepresentation~$V_l$ of~$V^{\otimes 2l}$. Write $u_{ij}^l$ and $v_{ij}^l$ for the matrix coefficients of $U_l$ and $V_l$ with respect to the basis $\{\xi^l_l, ..., \xi_{-l}^l\}$.

\begin{lemma} \label{lem:weights}
We have $v_{ij}^l=u_{ij}^lx^{-l-j}$ for all $l \in \frac{1}{2}\Z_+$ and $i,j = l, l-1,..., -l$.
\end{lemma}

\bp
The proof is by induction on $l$. By~\eqref{eq:UvsV} the claim is true for $l =1/2$, and it is obviously true for $l=0$. For the induction step we use that $V_{l+1/2}$ can be identified with a subrepresentation of $V_{l}\otimes V_{1/2}$. By definition, the vector $\xi^s_k$ has weight $2k$ with respect to $U(1)\subset SU_q(2)$, that is, $U_s|_{U(1)}(\xi^s_k\otimes1)=(\iota\otimes\pi_{U(1)})(U_s)(\xi^s_k\otimes1)=\xi^s_k\otimes z^{2k}$.  It follows that $\xi^{l+1/2}_{j}$ must be a linear combination of $\xi^l_{j-1/2}\otimes\xi^{1/2}_{1/2}$ and $\xi^l_{j+1/2}\otimes\xi^{1/2}_{-1/2}$. Hence $v^{l+1/2}_{ij}$ is a linear combination of
$v^l_{i-1/2,j-1/2}v^{1/2}_{1/2,1/2}$, $v^l_{i+1/2,j-1/2}v^{1/2}_{-1/2,1/2}$, $v^l_{i-1/2,j+1/2}v^{1/2}_{1/2,-1/2}$ and $v^l_{i+1/2,j+1/2}v^{1/2}_{-1/2,-1/2}$. By the inductive assumption all these elements lie in $\C[SU_q(2)]x^{-l-j-1/2}$. As $\pi(v^{l+1/2}_{ij})=u^{l+1/2}_{ij}$, we conclude that $v^{l+1/2}_{ij}=u^{l+1/2}_{ij}x^{-l-j-1/2}$.
\end{proof}

For $t \in [0,1]$, consider $\C[U_q(2)]$ as a left $\C[U_q(2)]$-module by
\[a \blacktriangleright_t b = a_{(1)}b ( f_t* S(a_{(2)}) ) \]
for $a, b \in \C[U_q(2)]$, where $f_t$ is the quasi-character on $\C[U_q(2)]$ given by
$$
f_t(\alpha)=q^t,\quad f_t(\alpha^*)=q^{-t},\quad f_t(\gamma)=f_t(\gamma^*)=0,\quad f_t(x)=1
$$
and $f_t*a=(\iota\otimes f_t)\Delta(a)$. We may in addition view $\C[U_q(2)]$ as a $\C[SU_q(2)]$-module by
\[g \triangleright_t b = i(g_{(1)})b \, i(f_t * S(g_{(2)}))\]
for $g \in \C[SU_q(2)]$ and $b \in \C[U_q(2)]$.
\begin{lemma} \label{lem:left-actions}
For any $a, b \in \C[U_q(2)]$, we have
\[a \blacktriangleright_t b = \pi(a) \blacktriangleright_t b = \pi(a) \triangleright_t b.\]
\end{lemma}
\begin{proof}
Let us first note that $f_t * x^n = x^n$ for any $n\in\Z$. Hence the first equality is immediate, because $f_t*S(x^n) = x^{-n}$ and $x$ is central.

Let $i_{\T}\colon \C[\T] \to \C[U_q(2)]$ denote the inclusion given by $i_\T(w):=x$. Then, for any $g \in \C[SU_q(2)]$, we have
\[\Delta(i(g)) = i(g_{(1)})i_\T({g_{(2)}}^{(1)})\otimes i({g_{(2)}}^{(2)}), \quad S(i(g)) = i_\T(S(g^{(1)}))i(S(g^{(2)})), \]
where $\delta(g) = g^{(1)}\otimes g^{(2)}$ is the left action by $\T$ on $\C[SU_q(2)]$ given by $\delta:=(\pi_\T\otimes\iota)\Delta$, so $\delta(\alpha) = 1 \otimes \alpha$, $\delta(\gamma) = w \otimes \gamma$.
%Note that we are abusing notation by using the letter $S$ for three different antipodes.
Now we obtain
\begin{align*}
g \blacktriangleright_t b &=  i(g_{(1)})i_\T({g_{(2)}}^{(1)}) b \, f_t*\big(i_\T(S({g_{(2)}}^{(2)}))i(S({g_{(2)}}^{(3)}))\big) \\
&= i(g_{(1)}) b \, (f_t* i(S({g_{(2)}}^{(3)})) \, i_\T({g_{(2)}}^{(1)}S({g_{(2)}}^{(2)})) \\
&= i(g_{(1)}) b \, ( f_t* i(S({g_{(2)}}^{(2)})) ) \varepsilon({g_{(2)}}^{(1)}) \\
&= i(g_{(1)}) b \,i(f_t* S(g_{(2)})) = g \triangleright_t b,
\end{align*}
%Now, we may assume that $g_{(1)} \otimes g_{(2)} = \sum_k g_{(1)}^k \otimes g_{(2)}^k$ where $\delta(g_{(2)}^k) = w^k \otimes g_{(2)}^k$. Then
%\begin{align*}
%i(g) \blacktriangleright_t b &= \sum_k i(g_{(1)}^k)i_\T(w^k) b \, f_t*(i_\T(S(w^k))i(S(g_{(2)}^k))) \\
%&= \sum_k i(g_{(1)}^k) b \, f_t* i(S(g_{(2)}^k)) \, x^kx^{-k} = i(g_{(1)}) b \,i(f_t* S(g_{(2)}))
%\end{align*}
where we used that $x$ is central and $f_t$ is trivial on $\C[\T]$ in the second equality, and that $f_t*i(a)=i(f_t*a)$ for all $a\in\C[SU_q(2)]$ in the fourth equality. %As $\pi$ is surjective, this is all we need.
\end{proof}

Recall that for a compact quantum group $G$, a vector space $Y$ with a left $\C[G]$-comodule structure $\delta\colon Y\to\C[G]\otimes Y$, $y\mapsto y^{(1)}\otimes y^{(2)}$, and a left $\C[G]$-module structure $\triangleright\colon \C[G]\otimes Y\to Y$ is called a Yetter--Drinfeld $G$-module, or a $D(G)$-module, if
$$
\delta(a\triangleright y)=a_{(1)}y^{(1)}S(a_{(3)})\otimes a_{(2)}\triangleright y^{(2)}\quad\text{for all}\quad a\in\C[G],\ y\in Y.
$$
The space $\C[U_q(2)]$ with the comodule structure $\Delta$ and module structure $\blacktriangleright_t$ is an example of a Yetter--Drinfeld $U_q(2)$-module.

\smallskip

Next, for $(m,n) \in \Z \oplus \Z=\hat T$, we consider the equivariant line bundle
$$
\mathcal{E}_{m,n} = U_q(2) \times_T \C_{m,n}
$$
over~$S^2_q$. To be precise, we look at the spaces of sections defined as follows:
\begin{align*}
 \Gamma(\mathcal{E}_{m,n}) &= \{ a \in \C[U_q(2)] \, : \, (\iota\otimes \pi_T)\Delta_{U_q(2)}(a) = a \otimes w^mz^n  \}\\
&=\operatorname{span}\{u^l_{i,n/2}x^m\, : \,l\in\frac{|n|}{2}+\Z_+,\ i=l,l-1,\dots,-l\}.
\end{align*}
Define also the $L^2$-sections by
\[\HH_{m,n} := \overline{\Gamma(\mathcal{E}_{m,n})}^{\|\cdot\|_2} \subset L^2(U_q(2)).\]
These subspaces are obviously invariant under the left regular representation of $U_q(2)$.

We have
\[\Gamma(\mathcal{E}_{0,n}) = \Gamma(\mathcal{E}_n)  \subset \C[SU_q(2)] \subset \C[U_q(2)],\]
where $\Gamma(\mathcal{E}_n)$ is the space considered in \cite[Section~4]{Voigt-BC}.
It follows from Lemma~\ref{lem:left-actions} and \cite[Lemma 4.1]{Voigt-BC} that
\[\C[U_q(2)] \to \mathrm{End}(\Gamma(\mathcal{E}_{0,n})), \quad a \mapsto a \blacktriangleright_1 (-), \]
extends to a $*$-homomorphism $C(U_q(2)) \to B(\HH_{0,n})$. Since the multiplication by $x^m$ defines a unitary isomorphism $\HH_{0,n}\cong\HH_{m,n}$ and is a $\C[U_q(2)]$-module map, it follows that $\blacktriangleright_1$ defines a representation of $C(U_q(2))$ on $\HH_{m,n}$ for all $m$ and $n$. Together with the restriction of the left regular representation of $U_q(2)$ this gives us a unitary representation of the Drinfeld double $D(U_q(2))$ on $\HH_{m,n}$. (The same is true by the obvious analogue of \cite[Lemma 4.1]{Voigt-BC} for $U_q(2)$, without relying on Lemma~\ref{lem:left-actions}.)

We also need the following definition:
\[C(\mathcal{E}_{m,n}) := \overline{\Gamma(\mathcal{E}_{m,n})}^{\|\cdot \|} \subset C(U_q(2)). \]
This is a $D(U_q(2))$-equivariant C$^*$-Hilbert $C(S^2_q)$-module, with the $U_q(2)$-action given by left translations and the $\C[U_q(2)]$-module structure given by the adjoint action $\blacktriangleright_0$.

\smallskip

We are finally ready to define some Kasparov modules. We construct an element
\[ \tilde{D} = [(\HH_{0,1} \oplus \HH_{1,-1}, \phi\oplus\phi,\tilde F)] \in KK^{D(U_q(2))}(C(S^2_q), \C) .\]
Here $\phi$ comes from the inclusion $C(S^2_q) \subset B(L^2(U_q(2))$ and the operator $\tilde F$ is defined by
\[ \tilde{F} := \begin{pmatrix}
1 & 0 \\
0 & x
\end{pmatrix}\begin{pmatrix}
0 & S \\
S^{-1} & 0
\end{pmatrix}\begin{pmatrix}
1 & 0 \\
0 & x^*
\end{pmatrix} \colon \HH_{0,1} \oplus \HH_{1,-1} \to \HH_{0,1} \oplus \HH_{1,-1}, \]
where $x\colon \HH_{0,-1} \to \HH_{1,-1}$ acts by multiplication and $S\colon \HH_{0,1} \to \HH_{0,-1}$ is determined by
$$
S(e_{i,-1/2}^l) := e_{i,1/2}^l,
$$
where $e^l_{ij}=u^l_{ij}/\|u^l_{ij}\|_2$. The map $\tilde F$ is $U_q(2)$-equivariant, since by Lemma~\ref{lem:weights} the map
$$
u^l_{ij}x^m=v^l_{ij}x^{m+l+j}\mapsto v^l_{i,j+1}x^{m+l+j}=u^l_{i,j+1}x^{m-1}
$$
is so. Since $\begin{pmatrix}
0 & S \\
S^{-1} & 0
\end{pmatrix}$ is the operator $F$ considered in~\cite{Voigt-BC}, it follows from Lemma~\ref{lem:left-actions} (for $t=1$) and results of~\cite[Section~4]{Voigt-BC} that $\tilde F$ is also $\C[U_q(2)]$-equivariant. We conclude that~$\tilde{F}$~is $D(U_q(2))$-equivariant and we indeed get an equivariant Kasparov module.

The $D(U_q(2))$-equivariant C$^*$-Hilbert $C(S^2_q)$-module $C(\mathcal{E}_{m,n})$ defines an element
\[ [[\mathcal{E}_{m,n}]] \in KK^{D(U_q(2))}(C(S^2_q),C(S^2_q)).\]
Put $\tilde{D}_k = [[\mathcal{E}_{0,k}]] \hat{\otimes} [\tilde{D}]$ and $[\mathcal{E}_{m,n}] = [u] \hat{\otimes} [[\mathcal{E}_{m,n}]]$, where $[u] \in KK^{D(U_q(2))}(\C, C(S^2_q))$ is induced by the inclusion $u\colon\C\to C(S^2_q)$.

\begin{proposition}
\label{prop:KK-Podles}
The classes $(-[\mathcal{E}_{-1,1}]) \oplus [\mathcal{E}_{0,0}]\in KK^{D(U_q(2))}(\C\oplus\C,C(S^2_q))$ and $[\tilde{D}]\oplus [\tilde{D}_{-1}]\in KK^{D(U_q(2))}(C(S^2_q),\C\oplus\C)$ are inverse to each other.
\end{proposition}
\begin{proof}
This is proved by following the arguments in \cite[Theorems 4.6 and~4.7]{Voigt-BC}. The key technical part there is the construction of a continuous family of representations $\omega_t\colon C(SU_q(2)) \to B(\HH_{0,0})$, $0\le t\le1$, by rescaling matrix coefficients of the representation $$\C[SU_q(2)] \to \mathrm{End}(\Gamma(\mathcal{E}_{0,0}))$$ defined by $\triangleright_t$. Combined with the representation of $SU_q(2)$ on $\HH_{0,0}$ this gives complementary series representations of $D(SU_q(2))$. Up to a $U_q(2)$-equivariant unitary equivalence, $\omega_t$ is obtained by keeping $\triangleright_t$ intact but changing the scalar product on $\Gamma(\E_{0,0})$. By Lemma~\ref{lem:left-actions} it follows then that if we extend~$\omega_{t}$ to $C(U_q(2))$ by setting $\omega_{t}(x) = 1$, then this extension together with the representation of $U_q(2)$ on $\HH_{0,0}$ define a unitary representation of~$D(U_q(2))$.

It remains to check that all $SU_q(2)$-equivariant maps used in the proof give rise to $U_q(2)$-equivariant maps. Heuristically, the reason why this should work is that, for every $m,n\in\Z$, the representations of $U_q(2)$ and $SU_q(2)$ on $\HH_{m,n}$ define the same images of $C^*(U_q(2))$ and $C^*(SU_q(2))$ in $B(\HH_{m,n})$, so every $SU_q(2)$-equivariant operator on $\HH_{m,n}$ is automatically $U_q(2)$-equivariant.
\end{proof}

From this we then get the following main result.

\begin{theorem}
For every $0<q<1$, the localizing subcategory of the triangulated category $KK^{U_q(2)}$ generated by the separable C$^*$-algebras with trivial $U_q(2)$-action coincides with $KK^{U_q(2)}$.
\end{theorem}
\begin{proof}
The proof follows the lines of \cite[Theorem 6.1]{Voigt-BC}, which establishes the same result for~$SU_q(2)$. The group $\hat{T} \cong \Z \oplus \Z$ is torsion-free and satisfies the strong Baum--Connes conjecture, and $C(U_q(2)/T) \cong \C \oplus \C$ in $KK^{D(U_q(2))}$ by the previous proposition. The corresponding properties for $U(1)$ and $C(SU_q(2)/U(1))$ are the essential ingredients used in the proof of \cite[Theorem 6.1]{Voigt-BC}. Thus the same proof works in our setting with $T$ instead of $U(1)$ and $U_q(2)$ instead of $SU_q(2)$.
\end{proof}

The result is known to be true for $q=1$ as well, see~\cite[Proposition~3.3]{MN}.

\bigskip

\printbibliography
\bigskip

\end{document}